\documentclass[11pt,reqno]{amsart} 

\usepackage[hmarginratio={1:1},scale=0.8]{geometry} 

 \usepackage[bookmarks=true,
bookmarksnumbered=true, breaklinks=true,
pdfstartview=FitH, hyperfigures=false,
plainpages=false, naturalnames=true,
colorlinks=false,
pdfpagelabels]{hyperref}

\usepackage{graphicx}%
\usepackage{multirow}%
\usepackage{amsmath,amssymb,amsfonts}%
\usepackage{amsthm}%
\usepackage{mathrsfs}%
\usepackage[title]{appendix}%
\usepackage{xcolor}%
\usepackage{textcomp}%
\usepackage{manyfoot}%
\usepackage{booktabs}%
\usepackage{algorithm}%
\usepackage{algorithmicx}%
\usepackage{algpseudocode}%
\usepackage{listings}%

\usepackage{mathtools}
\usepackage{enumerate}
\mathtoolsset{showonlyrefs}
\usepackage{comment}
\numberwithin{equation}{section}

\newtheorem{lemma}{Lemma}[section] 
\newtheorem{theorem}[lemma]{Theorem}

\newtheorem{definition}[lemma]{Definition}

\newtheorem{proposition}[lemma]{Proposition}

\makeatletter
\newtheorem*{rep@theorem}{\rep@title}
\newcommand{\newreptheorem}[2]{%
\newenvironment{rep#1}[1]{%
 \def\rep@title{#2 \ref{##1}}%
 \begin{rep@theorem}}%
 {\end{rep@theorem}}}
\makeatother
\newreptheorem{theorem}{Theorem}
\newreptheorem{corollary}{Corollary}

\newcommand{\R}{\mathbb{R}} 
\def \Rn{{\R^n}}
\newcommand{\hid}{m}
\def \Rnhi{{\R^{n\hid}}}
\newcommand{\e}{\varepsilon}

\newcommand{\N}{\mathbb{N}}
\def \E{\mathbb E}

\def \s{\mathbb{S}^{n-1}}
\def \S{\mathbb{S}^{n\hid-1}}

\def \conv{\operatorname{conv}}

\newcommand \B[1][n]{B_2^{#1}}
\newcommand{\vol}[1][n]{\operatorname{vol}_{#1}}
\newcommand{\Vol}{\operatorname{vol}_{nm}}
\def \s{\mathbb{S}^{n-1}}
\def \S{\mathbb{S}^{nm-1}}

\newcommand {\conbod}[1][nm] {\mathcal{K}^{#1}}
\newcommand {\conbodo}[1][nm] {\mathcal{K}^{#1}_o}
\newcommand {\conbodio}[1][nm] {\mathcal{K}^{#1}_{(o)}}

\newcommand{\PP}[1][Q]{\Pi_{#1,p}^{\circ}}

\newcommand{\D}{\Delta_{Q,p}^\mathcal{A}}
\renewcommand{\E}{\mathcal{E}_{Q,1}}
\newcommand{\Ep}{\mathcal{E}_{Q,p}}
\renewcommand{\L}{\lambda_{1,p}^\mathcal{A}(Q,\Omega)}
\newcommand{\Lq}{\lambda_{1,p}^{\mathcal{A},q}(Q,\Omega)}
\newcommand{\Rq}{\mathcal{R}^\mathcal{A}_{Q,p}}
\newcommand{\Rqq}{\mathcal{R}^\mathcal{A}_{Q,p,q}}
\newcommand{\Lp}{\|h_Q(\theta^t \nabla f(\cdot))\|_{L^p(\R^n)}}

\title[The $m$th-order Affine $p$-Laplace Operator and Applications]{Affine isoperimetric inequalities for the first eigenvalue of the $m$th order Affine $p$-Laplace Operator}

\author[D. Langharst]{Dylan Langharst}
\address{Department of Mathematical Sciences, Carnegie Mellon University, Pittsburgh, PA 15213, USA}
\email{dlanghar@andrew.cmu.edu}

\author[Roysdon]{Michael Roysdon}
\address{Department of Mathematics, The Ohio State University, Columbus OH, USA}
\email{roysdon.3@osu.edu}


\begin{document}
\begin{abstract}
Recently, Haddad, Jim\'enez, and Montenegro introduced the affine $p$-Laplace operator, $p>1$, and studied associated affine versions of the isoperimetric inequalities for the first eigenvalue of the affine $p$-Laplace operator, including the affine Faber-Krahn inequality and affine Talenti inequality. In this work, we introduce the $m$th-order $p$-Laplace operator $\D f$, which recovers the affine $p$-Laplace operator when $m=1$ and $Q$ is a symmetric interval.

Given $n,m \in \N$, a sufficiently smooth convex body $Q \subset \R^m$, a bounded, open set $\Omega \subset \R^n$ and $p >1$, we investigate the eigenvalue problem 
\[\begin{cases}
\D f = \L |f|^{p-2} f &\text{ in } \Omega;
\\
f=0 & \text{ on } \partial \Omega,
\end{cases}
\]
for $f \in W^{1,p}_0(\Omega)$. Finally, we establish $m$th-order extensions of the affine Talenti inequality and affine Faber-Krahn inequality, which, upon choosing $m=1$, yield new {\bf asymmetric} versions of those aforementioned inequalities.
\end{abstract}

\maketitle

\section{Introduction}
Over the past 50 years, there has been great interest in establishing affine-invariant versions of inequalities appearing in geometry and analysis that are far stronger than their classical counterparts. Two in particular are the affine Sobolev inequality and affine P\'olya-Szeg\H{o} principle.
In the seminal works \cite{GZ99,LYZ02,TW12}, the following \textit{affine $L^p$ Sobolev inequality} was shown: for $1\leq p <n$
\begin{equation}\label{eq 8.25.5}
	a_{n,p} \|f\|_{L^{\frac{np}{n-p}}(\Rn)}\leq \mathcal{E}_p(f).
\end{equation}
Here,
\begin{equation} \label{eq:sobolev_cons}
a_{n,p}=n^\frac{1}{p}\left(\frac{n-p}{p-1}\right)^\frac{p-1}{p}\left[\frac{\omega_n}{\Gamma(n)}\Gamma\left(\frac{n}{p}\right)\Gamma\left(n+1-\frac{n}{p}\right)\right]^\frac{1}{n}, \quad a_{n,1}=\lim_{p\to 1^+}a_{n,p},
\end{equation}
is the sharp Aubin-Talenti $L^p$ Sobolev constant and $\mathcal{E}_p(f)$ is the affine $L^p$ Sobolev energy of $f$, where $f$ belongs to $W_0^{1,p}(\R^n)$ for $p>1$ and $\operatorname{BV}(\R^n)$ for $p=1$. The equality conditions are also explicit. We suppress the definition of $\mathcal{E}_p(f)$; it can be found in \cite{LYZ02}, or, by substituting $m=1$ and $Q=[-1/2,1/2]$ in \eqref{eq:m_lp_affine_sobolev} below. However, we note that $\mathcal{E}_p(f) \leq \||\nabla f|\|_{L^p(\R^n)}$, and therefore \eqref{eq 8.25.5} implies the sharp Euclidean $L^p$ Sobolev inequality. Of course, we are using the usual notation for $L^p$ norms: letting $\Omega\subseteq \R^n$ be a Borel set, we have
\[
\|f\|_{L^{p}(\Omega)} = \left(\int_\Omega |f(x)|^pdx\right)^\frac{1}{p} \quad \text{and} \quad \||\nabla f|\|_{L^{p}(\Omega)} = \left(\int_\Omega |\nabla f(x)|^pdx\right)^\frac{1}{p}.
\]

Shortly thereafter, an affine {P}\'olya-{S}zeg\H{o} principle for the $L^p$ affine energy was established by Lutwak, Yang, and Zhang \cite{LYZ02} and Cianchi, Lutwak, Yang, and Zhang \cite{CLYZ09}:
\begin{equation}\label{eq 8.25.6}
	\mathcal{E}_p(f)\geq \mathcal{E}_p(f^\star) = \||\nabla f^\star|\|_{L^p(\R^n)}, \qquad p\geq 1,
\end{equation}
where $f^\star$ is the spherically symmetric rearrangement of $f$. See Section \ref{s:prelim} for its precise definition. An asymmetric (and even stronger) version of \eqref{eq 8.25.5} is due to Haberl, Schuster, and Xiao \cite{HSX12}. The characterization of equality for $p>1$ was found in Nguyen \cite{NVH16}. Like in the classical case, the inequalities leave the realm of Sobolev spaces and extend to functions of bounded variation, $\operatorname{BV}(\R^n)$, when $p=1$; Wang \cite{TW13} established the validity of \eqref{eq 8.25.5}, as well as the case of equality, in this instance.

The powerful affine $L^p$ Sobolev inequality and affine {P}\'olya-{S}zeg\H{o} principle became foundational tools in a burgeoning theory of affine PDEs, as initiated by Haddad, Jim\'enez and Montenegro \cite{HJM21}, and further explored in the sequel works \cite{HX22,LM24,LM24_2,LM24_3,HM25}. Recently, the named authors, working with Haddad, Putterman and Ye, established an $m$th-order affine $L^p$ Sobolev inequality in \cite{HLPRY25,HLPRY25_2}, and then the associated {P}\'olya-{S}zeg\H{o} principle with Zhao \cite{LRZ24}. With the foundational tools having been established, the next natural step is to consider $m$th-order affine PDEs, which include the aforementioned works when $m=1$.

To give some background, we recall the classical Euclidean case. The original $L^p$ Poincar\'e inequality \cite{HP90}, shown by Poincar\'e himself, states that, if $\Omega\subset \R^n$ is a bounded, open set, then for any $p\geq 1$ and $f \in C_0^\infty(\Omega)$, there exists $C(\Omega)>0$ such that
\begin{equation}
\label{eq:weak_poincare}
C(\Omega) \leq \mathcal{R}_p f,
\end{equation}
where $ \mathcal{R}_p f$ is the \textit{Rayleigh quotient} of $f$, given by
\[
 \mathcal{R}_p f: = \frac{\int_\Omega|\nabla f|^p dx}{\int_\Omega |f|^p dx}.
\]
The inequality \eqref{eq:weak_poincare} readily extends to $W_0^{1,p}(\Omega)$, which is the completion of $C_0^\infty(\Omega)$ with respect to the norm 
\[
\|f\|_{W_0^{1,p}(\Omega)}=\||\nabla f|\|_{L^p(\Omega)}.
\]
The optimal constant in \eqref{eq:weak_poincare} is the so-called \textit{Poincar\'e constant}:
\begin{equation}
\label{eq:Euclidean_poincare_constant}
    \lambda_{1,p}(\Omega):= \inf\left\{\mathcal{R}_p f \colon f \in W^{1,p}_0(\Omega) \setminus \{0\} \right\}.
\end{equation}

The fact that the infimum is obtained at some function $f_p \in W_0^{1,p}(\Omega)$ is easy to see, as $W_0^{1,p}(\Omega)$ embeds compactly in $L^p(\Omega)$. More surprising (to us) is the fact that the Poincar\'e constant is the first eigenvalue of the $p$-Laplacian operator 
\[
\Delta_p f := -\operatorname{div}(|\nabla f|^{p-2}\nabla f),
\]
and, therefore, $f_p$ is a bounded first eigenfunction of $\Delta_p$ on $\Omega$, i.e. it solves
\begin{equation}
\label{eq:euclidean__PDE}
\begin{cases}
\Delta_p f = \lambda_{1,p}(\Omega)|f|^{p-2}f &\text{in } \Omega,\\
f = 0 &\text{on } \partial \Omega.
\end{cases}
\end{equation}
Additionally, the function $f_p$ can be assumed to be positive (see e.g. \cite{PS04}) and is in $C^{1,\alpha}(\Omega)$ for some $\alpha$. If $\Omega$ is smooth, then this regularity extends up to the boundary, i.e. $f\in C^{1,\alpha}(\overline\Omega)$; see \cite{KU77, Tolksdorf, LG88}. When $p=1$, one should replace $W_0^{1,1}(\Omega)$ with $\operatorname{BV}(\Omega)$, as it is in this larger space that the infimum is obtained by $f_1\in \operatorname{BV}(\Omega)$; see \cite{KS07} and the references therein.

We mention now two inequalities associated with the operator $\Delta_p$. Let $f\in W_0^{1,p}(\R^n)$. The superlevel set of $f$ on $\Omega$ at level $t>0$ is 
\[
\Omega(t,f)=\{x\in\Omega: |f(x)| >t\}.
\]
Then, the \emph{distribution function} of $f$ is precisely
\begin{equation}
\label{eq:distribution_function}
\mu_f(t)=\vol[n]\left(\Omega(t,f)\right),\end{equation}
which is finite for $t>0$, since $f\in L^p(\R^n)$.
The $L^p$ Talenti inequality from \cite{GT79} yields an inequality for the distribution function of $f_p$, the unique solution to \eqref{eq:euclidean__PDE} up to dilation, when $1<p<n$. Letting $\mu_p(t):=\mu_{f_p}(t)$, it says that
\begin{equation}
    \left(n\omega_n^\frac{1}{n}(\mu_p(t))^\frac{n-1}{n}\right)^\frac{p}{p-1} \leq - \mu_p^\prime(t)\left(\lambda_{1,p}(\Omega)\left(\frac{\mu_p(t)}{t^{1-p}}+(p-1)\int_t^{+\infty}\frac{\mu_p(\tau)}{\tau^{2-p}}d\tau\right)\right)^\frac{1}{p-1}.
    \label{eq:Talenti}
\end{equation}

The second inequality concerns the Poincar\'e constant. The first eigenvalue of the Laplacian satisfies the following isoperimetric inequality, the so-called Faber-Krahn inequality: let $n\geq 2$ and let $\Omega\subset  \R^n$ be a bounded, open set. Then, for every $p\geq 1$,
\begin{equation}
\label{eq:FK_og}
\lambda_{1,p}(\Omega) \geq \lambda_{1,p}(B_\Omega),
\end{equation}
where $B_\Omega$ is the centered Euclidean ball with the same volume (Lebesgue measure) as $\Omega$. There is equality if and only if $\Omega$ is a translate of $B_\Omega$. See \cite{EK25} for $p=2$, \cite{AFT98} for $p>1$ and \cite{FMP09} for $p=1$.

Now that the classical background has been recalled, let us discuss the affine $L^p$ Laplacian $\Delta^\mathcal{A}_p$ introduced by Haddad, Jim\'enez and Montenegro, \cite{HJM21}; we remark that the case $p=2$ had been previously considered by Schindler and Tintarev \cite{ST18}.  We omit the precise formulations of their definitions and results; it will be clear below that their setting is recovered as a special case of that considered herein (with $m=1$ and $Q=[-\frac{1}{2},\frac{1}{2}]$ in Definition~\ref{d:D} and the results thereafter). The main philosophical advantage that the operator $\Delta^\mathcal{A}_p$ has over the usual Laplacian is its affine invariance: for $T\in \operatorname{SL}(n)$ and $f\in W_0^{1,p}(\Omega)$,
\[
\Delta_p^\mathcal{A} (f\circ T) = (\Delta_p^\mathcal{A} f)\circ T, \quad \text{on } \; T^{-1}(\Omega).
\]
Nevertheless, for $p>1$, $\Delta_p^\mathcal{A}$ agrees with the usual $p$-Laplacian $\Delta_p$ when applied to radial functions. Beyond the definition, Haddad, Jim\'enez and Montenegro \cite{HJM21} had established, in an impressive \textit{tour de force}, analogues of \eqref{eq:weak_poincare} and \eqref{eq:FK_og}; they showed that the affine version of the PDE \eqref{eq:euclidean__PDE} admits a solution that, furthermore, obtains the infimum in the affine version of the Poincar\'e constant \eqref{eq:Euclidean_poincare_constant}. Later, Haddad and Xiao \cite{HX22} continued this study, establishing the affine version of the Talenti inequality \eqref{eq:Talenti}.

We now delve into the $m$th-order framework which will set up our problem. Recall that a set $Q\subset\R^m$ is said to be a \textit{convex body} if it is compact, convex and has non-empty interior. We will use three collections of convex bodies $\conbod[m] \supseteq \conbodo[m]\supseteq \conbodio[m]$, which are respectively, the collection of convex bodies in $\R^m$, those that additionally contain the origin, and those that additionally contain the origin in their interiors. 

The $m$th-order Brunn-Minkowski theory initiated in \cite{HLPRY25,HLPRY25_2,LPRY25,LX24,LRZ24} incorporates a convex body $Q$ into operators from the $L^p$ Brunn-Minkowski theory introduced by Lutwak \cite{LE93,LE96}, Lutwak and Zhang \cite{LZ97}, and Lutwak, Yang and Zhang \cite{LYZ00,LYZ02,LYZ04,LYZ04_2}. These operators, which in the $m=1$ case would produce convex bodies in $\R^n$ from those in $\R^n$, now produce convex bodies in $\R^{nm}$ from those in $\R^n$. Specifically, when $m=1$, many of the operators contain $|\langle \theta,u \rangle|$, where $|\cdot|$ is the absolute value and $\langle \theta,u\rangle=\theta^tu$ is the (Euclidean) inner-product of two vectors $\theta,u\in \s$. Recalling that the support function of $K\in\conbod[n]$ is $h_K(u)=\sup_{y\in K}\langle y,u\rangle$, we observe that 
\[
|\langle \theta,u \rangle| = h_{[-1,1]}(\theta^tu).
\]
Then, we replace $[-1,1]$ with an $m$-dimensional convex  body $Q\in\conbodo[m]$, and $\theta$ becomes an element of $(\R^n)^m=\R^{nm}$, i.e., we obtain terms of the form $h_Q(\theta^tu)$, where now $\theta$ is viewed as a $n\times m$ matrix whose column vectors are $\theta_i\in\R^n,$ $i=1,\dots,m$. Note that we are suppressing the dependence on $Q$ (and later $p$ as well) in the term ``$m$th-order'' for expositional purposes.

Returning to the $m=1$ case, we remark that Haberl and Schuster in \cite{HS09,HS2009,HJS21} saw this case to completion (with respect to affine $L^p$ inequalities of the geometric-and-Sobolev-type) by replacing $|\langle \theta,u \rangle|$ with $(1-\tau)\langle \theta,u \rangle_-+\tau\langle \theta,u \rangle_+$, where $\tau \in [0,1]$. Here, for $a\in\R$, we have as usual $a_{\pm}= \max\{0,\pm a\}$. In our notation, this is precisely, up to dilation, $Q=[-\alpha,\beta]$ for $\alpha,\beta>0$. Inequalities from convex geometry concerning operators generated by this type of $Q$ are said to be \textit{asymmetric}. For those not interested in $m>1$, our results also serve to complete the $m=1$ case of Haddad, Jim\'enez and Montenegro \cite{HJM21}, as our setting includes Haberl and Schuster's asymmetric operators as well.  

We are now in a position to introduce the $m$th-order affine Sobolev energy $\mathcal{E}_{Q,p}$: for $f \in W_0^{1,p}(\R^n)$ (or in $\operatorname{BV}(\R^n)$ if $p=1$), and  $Q \in \conbodo[m]$, set
\begin{equation}
\label{eq:m_lp_affine_sobolev}
\Ep f = d_{n,p}(Q) \left( \int_{\S}\|h_Q (\theta^t\nabla f(\cdot))\|_{L^p(\R^n)}^{-nm}d\theta \right)^{- \frac{1}{nm}},
\end{equation}
where 
\[
d_{n,p}(Q) = (n\omega_n)^{\frac{1}{p}} (nm \Vol(\Pi_{Q,p}^\circ B_2^n))^{\frac{1}{nm}}.
\]
The $m$th-order polar projection operator $\Pi_{Q,p}^\circ$ is one of the aforementioned operators that maps a convex body in $\R^n$ to one in $\R^{nm}$. The $m$th-order polar projection body of the unit Euclidean ball $B_2^n$, $\Pi_{Q,p}^\circ B_2^n \in \conbodo[nm]$, is such a convex body. See Definition~\ref{d:generalprojectionbody} below.  We recall that
\[
\|h_Q (\theta^t\nabla f(\cdot))\|_{L^p(\R^n)} = \left(\int_{\R^n}h_Q (\theta^t\nabla f(x))^pdx\right)^\frac{1}{p}.
\]
When $p=1$ and $f$ belongs to $\operatorname{BV}(\R^n)$ instead of $W_0^{1,p}(\R^n)$, the quantity $\|h_Q (\theta^t\nabla f(\cdot))\|_{L^p(\R^n)}$ should be understood as
\[
\int_{\R^n}h_Q (\theta^t\eta_f(x))d|Df|(x),
\]
where $|Df|$ is the variation measure of $f$ and $\eta_f$ is the Radon-Nikodym derivative of $Df$ with respect to $|Df|$, and where $Df$ is a vector-valued measure depending on $f$ whose precise definition is not necessary for our investigations herein. Like in the $m=1$ case, $\Ep f$ satisfies $\Ep f \leq \||\nabla f|\|_{L^p(\R^n)}$ (see \cite[Theorem 1.2]{LRZ24}).

Our main focus is the following minimization problem.

\begin{definition} Fix $m,n\in\N$, $p \geq 1$ and $Q\in\conbodo[m]$. For a bounded, open set $\Omega \subset \R^n$, define
\begin{equation}\label{e:poincareconst}
\L = \inf\left\{\Rq f \colon f \in W^{1,p}_0(\Omega) \setminus \{0\} \right\},
\end{equation}
where $W^{1,p}_0(\Omega)$ should be $\operatorname{BV}(\Omega)$ if $p=1$ and $\Rq f$ denotes the affine $m$th-order Rayleigh quotient given by 
\[
\Rq f = \frac{(\Ep f)^p}{\|f\|_{L^p(\Omega)}^p}.
\]
    
\end{definition}

To establish the fact that $\L$ is non-trivial (i.e., that it is greater than zero and achieved by some function), the first step is to proceed as Poincar\'e himself, and establish a Poincar\'e-type inequality for the affine $m$th-order energy functional $\Ep$ for \textit{some} constant. This was already done in \cite[Theorem 5.1]{LRZ24} (see Lemma~\ref{l:useful}(ii) below), and so we may continue to the next step.

Following \cite{BFK03}, given $K\in\conbodio[n]$, its Wulff $p$-Laplacian is given by 
\[
\Delta_{p,K}f(x) = -\text{div}\left(\nabla\left(\frac{1}{p} h_K^p\right)\left(\nabla f(x)\right)\right).
\]
In \eqref{eq:L_body} below, we introduce a particular convex body $L_{Q,p,f}$ depending on $f$. With this choice, we are now prepared for the following definition.

\begin{definition}
\label{d:D}
Given $p >1$, a bounded, open set $\Omega \subset \R^n$, and $Q\in\conbodo[m]$, we define the affine $m$th-order $p$-Laplace operator of $f \in W^{1,p}_0(\Omega)$ by 
\begin{equation}\label{e:laplacianidentity}
\D f = \Delta_{p,L_{Q,p,f}} f=-\operatorname{div}(h_{L_{Q,p,f}}^{p-1}(\nabla f) \nabla h_{L_{Q,p,f}}(\nabla f)).
\end{equation}
\end{definition}

Section~\ref{s:prelim} is dedicated to some background facts from convex geometry and the theory of PDEs. Section~\ref{s:PDEs} then discusses the operator $\D$ introduced in Definition~\ref{d:D}. Our first main result, which we prove in Section~\ref{sec:regality}, is the following relation between $\D$ and $\L$. Consider the following eigenvalue problem:
\begin{equation}\label{e:weak}
\begin{cases}
\D f = \L |f|^{p-2} f &\text{ in } \Omega;
\\
f=0 & \text{ on } \partial \Omega.
\end{cases}
\end{equation}
We require the following class of convex bodies; these regularity assumptions on our convex bodies $Q$ allow us to differentiate when necessary:
\[
\conbod[m]_{\operatorname{s}} = \{Q \in \conbodio[m] \colon  h_Q \in C^2(\R^m \setminus \{0\})\}.
\]

Note that this class was essentially used both in \cite{HS09,HJM21} for the case $m=1$, where $Q$ is an interval (see \cite[Lemma~4.1]{HS09}). In particular, $\conbod[1]_{\operatorname{s}} =\conbodio[1]$.
\begin{theorem}\label{t:operatoranalysis} Fix $m,n\in\N$, $Q\in \conbod[m]_{\operatorname{s}}$ and $p>1$. Let $\Omega \subset \R^n$ be a bounded, open set. Then, $f_p \in W_0^{1,p}(\Omega)$ is an eigenfunction of $\D$ on $W^{1,p}_0(\Omega)$ corresponding to $\L$ in the sense of \eqref{e:weak} if and only if $f_p$ minimizes $\Rq$. In particular, we have the following two useful facts:
\begin{itemize}
    \item[(i)] $\L$ is the smallest among all the real Dirichlet eigenvalues of $\D$; 
    \item[(ii)] each eigenfunction corresponding to $\L$ is a bounded function in $C^{1,\alpha}(\Omega)$ for some $\alpha \in (0,1)$. Furthermore, the function is in $C^{1,\alpha}(\overline{\Omega})$ provided $\partial \Omega$ is of class $C^{2,\alpha}$. Additionally, if $\Omega$ is connected, each such eigenfunction does not change sign.
\end{itemize}
\end{theorem}
The proof of Proposition~\ref{p:eleq} reveals that $\D$ can be defined in terms of its action on $W_0^{1,p}(\Omega)$:
\[
\langle \D f,\varphi\rangle
=
\frac{1}{p}
\left.\frac{d}{dt}(\Ep(f+t\varphi))^p\right|_{t=0}
=
\int_\Omega
\left\langle
h_{L_{Q,p,f}}(\nabla f)^{p-1}
\nabla h_{L_{Q,p,f}}(\nabla f),
\nabla\varphi
\right\rangle dx.
\]

Recently, Crasta and Fragal\`a created \cite{CF20} a framework studying generalizations of \eqref{eq:euclidean__PDE} with the Laplacian $\Delta_p$ replaced by a more general operator. In particular, they showed that solutions are log-concave. Our operator $\D$ falls outside this framework, since the convex body $L_{Q,p,f}$ depends on $f$. For this reason, establishing a Brunn-Minkowski-type inequality for $\L$ and proving the simplicity of the eigenfunctions $f_p$ remain beyond our current framework.

We conclude the topic of \eqref{e:weak} by establishing the $m$th-order analogue of the $L^p$ Talenti inequality from \eqref{eq:Talenti}. This elaborates on the $m=1$, $Q=[-1,1]$ case from Haddad and Xiao \cite{HX22}. The proof of the following theorem is in Section~\ref{s:talenti}.
\begin{theorem}
\label{t:talenti}
    Let $1<p<n$ and $\Omega\subset \R^n$ be a bounded, open set, and $Q\in \conbod[m]_{\operatorname{s}}$. Let $f_p$ be the first eigenfunction of \eqref{e:weak}, or, equivalently, $f_p$ is a minimizer on $W_0^{1,p}(\Omega)\setminus\{0\}$ for $\Rq$. Then, its distribution function $(0,\infty)\ni t\mapsto \mu_p(t):=\mu_{f_p}(t)$ satisfies

    \begin{equation}
    \left(n\omega_n^\frac{1}{n}(\mu_p(t))^\frac{n-1}{n}\right)^\frac{p}{p-1} \leq - \mu_p^\prime(t)\left(\L\left(\frac{\mu_p(t)}{t^{1-p}}+(p-1)\int_t^{+\infty}\frac{\mu_p(\tau)}{\tau^{2-p}}d\tau\right)\right)^\frac{1}{p-1}.
    \label{eq:Talenti_mth}
\end{equation}
\end{theorem}

A natural question is, how does $\L$ compare to $\lambda_{1,p}(\Omega)$? Elaborating on \cite[Theorem 6]{HJM21}, the following proposition is an immediate consequence of \eqref{eq:relations} below.
\begin{proposition}
\label{p:relations}
Fix $p\geq 1,$ $m,n\in\N$, $Q\in\conbodio[m]$ and a bounded, open set $\Omega\subset \R^n$. Then, there exists a constant $d=d(Q,p,\Omega)$ such that
\[
\lambda_{1,p}(\Omega) \geq \lambda_{1,p}^{\mathcal{A}}(Q,\Omega) \geq d\left(\lambda_{1,p}(\Omega)\right)^\frac{1}{nm}.
\]
\end{proposition}

Our next result is the following affine $m$th-order Faber-Krahn inequality, which we prove in Section~\ref{sec:fk}. This elaborates on the case $m=1$ and $Q=[-1,1]$  from the influential work by Haddad, Jim\'enez, and Montenegro \cite{HJM21}. 
\begin{theorem}
\label{t:fk}
Fix $m,n\in\N$, $Q\in \conbod[m]_{\operatorname{s}}$, $p \geq 1$ and let $E \subset \R^n$ be an ellipsoid.  Then,
\begin{equation}\label{e:fkinequality}
\L \geq \lambda_{1,p}^{\mathcal{A}}(Q,E)
\end{equation}
holds for every bounded, open subset $\Omega$ in $\R^n$ having the same volume as $E$. Furthermore, there is equality in \eqref{e:fkinequality} if and only if $\Omega$ is an ellipsoid with the same volume as $E$. 
\end{theorem}

Finally, we recall the relationship between the Poincar\'e constant $\lambda_{1,1}(\Omega)$ and the so-called \textit{Cheeger} constant given by
\begin{equation}
\label{eq:Cheeger}
I_1(\Omega) = \inf_{C\subseteq \Omega}\frac{\vol[n-1](\partial^\ast C)}{\vol(C)}.
\end{equation}
Here, $\partial^\ast C$ denotes the reduced boundary of $C$, and the infimum runs over all sets of finite perimeter (recall, this means $\chi_C\in \operatorname{BV}(\Omega)$). It is well known that the infimum is obtained by some set $C_0$, a so-called Cheeger set of $\Omega$, and that $\chi_{C_0}$ minimizes $\lambda_{1,1}(\Omega)$. In fact, $I_1(\Omega) = \lambda_{1,1}(\Omega)$. 

In \cite[Theorem 8]{HJM21}, the authors of that work considered the affine analogue of Cheeger sets. A key finding of theirs was that these sets are in maximal volume position. Given a compact $C\subset \R^n$, a position of $C$ is precisely an affine image of $C$, i.e. a set of the form $AC+x_0$, where $A\in \operatorname{GL}(n)$ and $x_0\in\R^n$. We say $C\subset \Omega$ is in position of maximal volume if $\vol(C) \geq \vol(AC+x_0)$ for any position $AC+x_0$ of $C$ such that $AC+x_0 \subset \overline{\Omega}$. Following their lead, we establish the following result, which we prove in Section~\ref{sec:cheeger}.

\begin{theorem}
\label{t:cheeger}
    Fix $m,n\in\N$ and let $\Omega\subset\R^n$ be a bounded, open set. Then, the $m$th-order affine Cheeger constant of $\Omega$ with respect to $Q\in\conbodio[m]$, given by 
    \[
    I_{Q,1}^\mathcal{A} = \inf_{C\subset \Omega} \frac{\E \chi_C}{\|\chi_C\|_{L^1(\R^n)}} = d_{n,1}(Q)(nm)^{-\frac{1}{nm}}\inf_{C\subset \Omega}\frac{\vol[nm](\Pi^\circ_{Q,1} C)^{-\frac{1}{nm}}}{\vol(C)},
    \]
    is minimized among all measurable sets of finite perimeter by a set $C_0$, referred to as an $m$th-order affine Cheeger set of $\Omega$ with respect to $Q$, for which $\chi_{C_0}$ is a minimizer of $\lambda_{1,1}^\mathcal{A}(Q,\Omega)$. Moreover, $C_0$ is in maximal volume position inside $\Omega$.
\end{theorem}

The set $\Pi^\circ_{Q,1} C$ is defined via its gauge in \eqref{eq:polar_projection_finite_perim} below, and we used \eqref{eq:gauge_volume} for its volume. 

\section{Preliminaries}\label{s:prelim}

\subsection{Function spaces}
Throughout this work, when $p=1$, we identify a function on a set $\Omega\subset \R^n$ with its zero extension to $\R^n$, and $\operatorname{BV}(\Omega)$ denotes those functions whose zero extensions belong to $\operatorname{BV}(\R^n)$.

We start with the following definitions for functions.
Given $f\in L^p(\R^n)$, its \emph{decreasing rearrangement} $f^\ast:[0,\infty)\rightarrow [0,\infty]$ is defined by
\begin{equation}
  f^\ast(s)=\sup\{t>0:\mu_f(t)>s\} \quad \text{for }s\geq 0.
\end{equation}
The \emph{spherically symmetric rearrangement} $f^\star$ of $f$ is defined as
\begin{equation}
\label{eq:rearrange}
    f^\star(x)=f^\ast(\omega_n|x|^n) \quad \text{for } x\in \R^n.
\end{equation}
Spherically symmetric rearrangement preserves the distribution function, that is
\begin{equation}
    \mu_{f}=\mu_{f^\star}.
\label{eq:dis_funs}
\end{equation}
We will require the following regularity conditions due to Tolksdorf
\cite{Tolksdorf}, which we state as a lemma.
\begin{lemma} \label{l:Tolksdorf}
Let $p>1$ and let $\nu:\R^n\to\R^n$ be a vector field for which there exist positive constants $c_1$ and $c_2$, independent of $z$ and $x$, such that, for every $z \in \R^n \setminus \{0\}$ and $x \in \R^n$,
\[
\sum_{j,k=1}^n
\frac{\partial \nu_j(z)}{\partial z_k}x_jx_k
\geq c_1 |z|^{p-2}|x|^2
\quad\text{and}\quad
\sum_{j,k=1}^n
\left|
\frac{\partial \nu_j(z)}{\partial z_k}
\right|
\leq c_2 |z|^{p-2}.
\]
Then the operator
\[
f\mapsto-\operatorname{div}(\nu(\nabla f))
\]
is a degenerate quasilinear elliptic operator.
\end{lemma}

\noindent We will refer to degenerate quasilinear elliptic operators of the form
\[
-\operatorname{div}\left(
H(\nabla f)^{p-1}\nabla H(\nabla f)
\right),
\]
where $H$ is a nonnegative, positively $1$-homogeneous convex function that is differentiable on $\R^n\setminus\{0\}$, as being of Wulff-type.

We now consider convex geometry. The \emph{gauge} of $K\in \conbodio[n]$ is defined to be 
\begin{equation}
	\|y\|_K=\inf\{r>0:y\in rK\}.
\end{equation}
Note that $\|\cdot\|_{K}$ is the possibly asymmetric norm whose unit ball is $K$. One has the formula
\begin{equation}
\label{eq:gauge_volume}
    \vol[n](K)=\frac{1}{n}\int_{\s}\|\theta\|_{K}^{-n}d\theta,
\end{equation}
where $d\theta$ represents the spherical Lebesgue measure on $\s$, the unit sphere.

We recall that the support functions $h_K$ are $1$-homogeneous. Therefore, we have Euler's identity for $K\in\conbod[n]$ and $u\in\s$ such that $\nabla h_K(u)$ exists:
\begin{equation}
\label{eq:support_homo}
h_K(u)=\frac{d}{dt}th_K(u)=\frac{d}{dt}h_K(tu) = \langle \nabla h_K(u),u \rangle.
\end{equation}

We next recall the $L^p$ surface area measure of a convex body $K$. For $p=1$, the surface area measure $\sigma_K=\sigma_{K,1}$ of a convex body $K\in\conbod[n]$, is given as follows. Denote by $n_K$ the Gauss map of $K$ defined almost everywhere on $\partial K$; its \emph{surface area measure}, $\sigma_K$, is the Borel measure on $\mathbb{S}^{n-1}$ given by
\[\sigma_K(D)=\mathcal H ^{n-1}(n^{-1}_K(D)),\]
for every Borel subset $D$ of $\s$, where $n_K^{-1}(D)$ consists of all boundary points of $K$ with outer unit normals in $D$. For $p>1$, $\sigma_{K,p}$ is the $L^p$ surface area measure of $K\in \conbodio[n]$, which was introduced by Lutwak \cite{LE93, LE96}:
\begin{equation}
	d\sigma_{K,p}(u)=h_K(u)^{1-p}d\sigma_{K}(u). 
\end{equation}
The $L^p$ surface area measure is a cornerstone of the $L^p$ Brunn-Minkowski theory and has been a major focus of research in the theory of convex bodies over the last few decades.

We can now state the following isoperimetric inequality for sets of finite perimeter from \cite[Proposition 2.3]{AFT97}. We say $E\subset\R^n$ is a set of finite perimeter with respect to an origin-symmetric convex body $K\subset\R^n$ if $P_{K}(E,\R^n) < + \infty,$
where
\begin{equation}
    P_{K}(E,\R^n) :=\sup\left\{\int_E\operatorname{div} \sigma(x)dx:\sigma\in C_0^1(\R^n,\R^n) \text{ and } \|\sigma\|_K\leq 1\right\}.
\end{equation}
Then, the following inequality holds:
\begin{equation}
\label{eq:min}
P_{K}(E,\R^n) \geq n\vol(K)^\frac{1}{n}\left(\vol(E)\right)^\frac{n-1}{n}.
\end{equation}
In the case when $E$ is a convex body, \eqref{eq:min} is nothing but Minkowski's first inequality for mixed volumes. Finally, we have the coarea formula for any nonnegative function $u\in W_{0}^{1,p}(\Omega)$: 
\begin{equation}
    \int_{\Omega} h_K(\nabla u(x)) d x=\int_0^{\infty} P_K\left(\{x \in \Omega: u(x)>s\}, \mathbb{R}^n\right) d s.
    \label{eq:coarea}
\end{equation}
It also holds for functions of bounded variation; the following is implied by Proposition 3.15 and Theorem 3.40 in \cite{AFP}. Recall that $f_{\pm}=\max\{0,\pm f\}$.
\begin{lemma}
    Let $f\in \operatorname{BV}(\R^n)$ and $h$ a real-valued, $1$-homogeneous function on $\R^{n}\setminus\{0\}$. Then, 
    \begin{equation}
    \label{eq:coarea_BV}
    \begin{split}
        \int_{\R^n}h(\eta_f(y))d|Df|(y) &= \int_0^{\infty}\int_{\partial^\ast \{x\in\R^n:f(x)>t\}} h(\eta_{f_+}(y))d\mathcal{H}^{n-1}(y)dt
        \\
        &\quad+\int_0^{\infty}\int_{\partial^\ast \{x\in\R^n:f<-t\}} h(-\eta_{f_-}(y))d\mathcal{H}^{n-1}(y)dt.
    \end{split}
    \end{equation}
\end{lemma}

\subsection{Results from the theory of $m$th-order convex bodies}
We are now ready to discuss concepts from the recent $m$th-order theory of convex bodies, which were introduced in \cite{HLPRY25,HLPRY25_2}. Each of our operators will feature $h_Q^p$ as a type of kernel. For future investigations, here we prove a smoothness lemma for this kernel.
\begin{lemma}
\label{l:kernel_smooth}
Fix $n,m\in \N$ and $p\geq 1$. Let $Q\in\conbodo[m]$. Let $R_Q$ be the radius of the smallest centered Euclidean ball containing $Q$. Then, for every $y\in \R^{nm}$:
            \begin{enumerate}
                \item[(a).] $|h_Q(y^t\xi)^p-h_Q(y^t\eta)^p| \leq pR_Q\max\{h_Q(y^t\xi)^{p-1},h_Q(y^t\eta)^{p-1}\}|y||\xi-\eta|$ for all $\xi,\eta \in \Rn$.
                \item[(b).] $|h_Q(y^t\xi)| \leq R_Q|y||\xi|$ for all $\xi \in \Rn$.
                \item[(c).] 
                For every $\xi \neq o$, there is a set $A(\xi) \subseteq \Rnhi$ satisfying $\Vol(\R^{nm}\setminus A(\xi))=0$ and, if $y \in A(\xi)$, then the function $\eta\mapsto h_Q(y^t\eta)^p$ is differentiable at $\xi$.
                \item[(d).] $|\nabla_\xi h_Q(y^t\xi)^p| \leq pR_Qh_Q(y^t\xi)^{p-1}|y|$ for $y \in A(\xi)$ and $\xi \neq o$.

            \end{enumerate}
\end{lemma}

\begin{proof}
    {\bf Proof of item (a). }
    Let $\|y^t\|_{\text{op}}$ be the operator norm of the matrix $y^t:\R^n\to \R^m$. Then, 
            \[
            |y^t(\xi-\eta)| \leq \|y^t\|_{\text{op}}|\xi-\eta| \leq |y||\xi-\eta|.
            \]
Next, we use the fact that $h_Q$ is sublinear, since it is convex and $1$-homogeneous. On the one hand,
\[
h_Q(y^t\xi)-h_Q(y^t\eta) \leq h_Q(y^t(\xi-\eta)) \leq R_Q|y^t(\xi-\eta)|,
\]
while, on the other hand,
\[
h_Q(y^t\eta)-h_Q(y^t\xi) \leq h_Q(y^t(\eta-\xi)) \leq R_Q|y^t(\xi-\eta)|.
\]
Consequently,
            \begin{equation}
            \label{eq:sublinear}
           |h_{Q}(y^t\xi)-h_{Q}(y^t\eta)|
           \leq R_Q|y^t(\xi-\eta)|
           \leq R_Q |y||\xi-\eta| .
            \end{equation}
            Recall that the function $t\mapsto t^p$ is locally Lipschitz: for every $t_1,t_2\geq0$, one has
            \begin{equation}
            \label{eq:lipschitz}
            |t_1^p-t_2^p| \leq p\cdot \max\{t_1^{p-1},t_2^{p-1}\}\cdot |t_1-t_2|.
            \end{equation}
            Applying \eqref{eq:lipschitz} to $t_1 = h_Q(y^t\xi)$ and $t_2=h_Q(y^t\eta)$, we have by \eqref{eq:sublinear}
            \begin{align*}
            |h_Q(y^t\xi)^p-h_Q(y^t\eta)^p| 
            &\leq p\cdot \max\{h_Q(y^t\xi)^{p-1},h_Q(y^t\eta)^{p-1}\}\cdot |h_Q(y^t\xi)-h_Q(y^t\eta)|
            \\
            &\leq  p\cdot \max\{h_Q(y^t\xi)^{p-1},h_Q(y^t\eta)^{p-1}\}\cdot R_Q |y||\xi-\eta|.
            \end{align*}
            We have shown item (a).

            {\bf Proof of item (b). } 
            Item $(b)$ follows from the argument in $(a)$ by taking $p=1$ and $\eta=o$.

             {\bf Proof of item (c). } 
            As for item (c), since $h_Q$ is Lipschitz, it is differentiable almost everywhere by Rademacher's theorem. Let us denote by $S_Q\subset \R^m$ the set of measure zero upon which $h_Q$ is not differentiable. For $\xi\neq o$, we define the surjective linear map $T_\xi:\R^{nm}\to\R^m$ by $T_\xi (y)=y^t\xi$. We set
            \[
            A(\xi)^c :=T_{\xi}^{-1}(S_Q)= \left\{y\in \R^{nm}:y^t\xi\in S_Q\right\}.
            \]
            The pre-image of a set of measure zero under a surjective linear map from a higher-dimensional space to a lower-dimensional space is also measure zero. Defining $A(\xi):=\R^{nm}\setminus A(\xi)^c$, we have $\Vol(\R^{nm}\setminus A(\xi))=0$. If $y\in A(\xi)$, then $h_Q$ is differentiable at $y^t\xi$, and hence the function
            \[
            \eta\mapsto h_Q(y^t\eta)^p
            \]
            is differentiable at $\xi$. This concludes the proof of item (c).

            {\bf Proof of item (d).}
            For $(d)$, by item $(c)$ the function $\eta\mapsto h_Q(y^t\eta)^p$ is differentiable at $\xi$ when $\xi\neq o$ and $y\in A(\xi)$. Thus, for $t>0$ small and $u\in\s$,
            \[
            h_Q(y^t(\xi+tu))^p=h_Q(y^t\xi)^p+t\langle \nabla_\xi h_Q(y^t\xi)^p,u \rangle +o(t),
            \]
            where $o$ is a function so that $o(t)/t \to 0$ as $t\to 0$. Therefore, from item (a),
            \begin{align*}
            \left|\langle \nabla_\xi h_Q(y^t\xi)^p,u \rangle +\frac{o(t)}{t}\right| &= \left|\frac{h_Q(y^t(\xi+tu))^p-h_Q(y^t\xi)^p}{t}\right| 
            \\
            &\leq pR_Q\max\{h_Q(y^t\xi)^{p-1},h_Q(y^t(\xi+tu))^{p-1}\}|y|.
            \end{align*}
            Sending $t\to 0$, we obtain
            \[
            |\langle \nabla_\xi h_Q(y^t\xi)^p,u\rangle| \leq p\cdot R_Q \cdot h_Q(y^t\xi)^{p-1}\cdot |y|.
            \]
            Taking supremum over all $u \in \s$, we obtain
            \[
            |\nabla_\xi h_Q(y^t\xi)^p| \leq pR_Qh_Q(y^t\xi)^{p-1}|y|,
            \]
            as claimed. 
\end{proof}
 
We now define our operators.

\begin{definition} \label{d:generalprojectionbody} Let $p \geq 1$, $m \in \N$, and fix $Q\in \conbodo[m]$. Given $K\in\conbodio[n]$, we define the $(L^p, Q)$- polar projection body of $K$, denoted $\PP K \in\conbodo[nm]$, to be the convex body whose gauge is given by 
\[
\|u\|_{\PP K} = \left(\int_{\s} h_Q(u^tv)^p d\sigma_{K,p}(v)\right)^\frac{1}{p}, \quad u\in\S.
\]
\end{definition}
This extends the $m=1$ cases from \cite{petty61_1,LYZ00,TW12,HS09}. We take a moment to discuss in more detail the $p=1$ case. By using the Gauss map, we can write
\[
\|\theta\|_{\Pi^\circ_{Q,1} K} = \int_{\partial K} h_Q(\theta^tn_K(y)) d\mathcal{H}^{n-1}(y).
\]

Using some aforementioned tools from the theory of functions of bounded variation, we can extend $\Pi^\circ_{Q,1}$ to sets of finite perimeter. A Borel set $D\subset \Rn$ is said to be \textit{a set of finite perimeter} if $\chi_D \in \operatorname{BV}(\Rn)$. Examples include convex bodies and compact sets with $C^1$ boundary. Given such a set $C$, we define its $m$th-order polar projection body with respect to $Q$ via the gauge
\begin{equation}\|\theta\|_{\Pi^\circ_{Q,1} C} = \int_{\partial^\ast C} h_Q(\theta^tn_C(y) )d\mathcal{H}^{n-1}(y), \quad \theta\in\S,
\label{eq:polar_projection_finite_perim}
\end{equation}
where $\partial^\ast C$ is the reduced boundary of $C$ and $n_C:=-\eta_{\chi_C}$ is its measure-theoretic outer-unit normal, see e.g. \cite[Definition 2.2]{TW12}.

We also require the following centroid body introduced in \cite{HLPRY25_2}: for a compact subset $L \subset \R^{nm}$ of positive volume the convex body $\Gamma_{Q,p} L \subset \R^n$ is the $m$th-order centroid body of $L$ given by 
\[
h_{\Gamma_{Q,p}L}(v) =\left( \frac{1}{\Vol(L)} \int_L h_Q(x^tv)^p dx\right)^\frac{1}{p}. 
\]
If, in addition, we assume that $L$ is star-shaped about the origin, then we have from polar coordinates
\begin{equation}
\label{eq:centroid_polar}
h_{\Gamma_{Q,p}L}(v) =\left(\frac{1}{nm+p}\frac{1}{\Vol(L)} \int_{\S} \|\theta\|_L^{-nm-p} h_Q(\theta^t v)^p d\theta\right)^\frac{1}{p}. 
\end{equation}
Let $T\in \operatorname{GL}(n)$. We introduce the notation, for $x\in \R^{nm}$,
\[
\bar T x = (Tx_1,\dots,Tx_m),
\]
which is nothing but the matrix multiplication of the $n\times n$ matrix $T$ and the $n\times m$ matrix $x$ whose columns are the vectors $x_i$. Then, the centroid body operator $\Gamma_{Q,p}$ has the following invariance:
\[
\Gamma_{Q,p} (\bar TL) = T\Gamma_{Q,p}L.
\]

The following Busemann-Petty inequality for the $m$th-order-centroid body was shown in \cite[Theorem 1.6]{HLPRY25_2}: for a compact set $L\subset\R^{nm}$ with positive volume, $p\geq 1$, and $Q\in\conbodo[m]$, it holds
\begin{equation}
\label{eq:BPcLQ}
\frac{\vol(\Gamma_{Q,p}  L)}{\Vol(L)^{\frac{1}{m}}} \geq \frac{\vol(\Gamma_{Q,p} \Pi^{\circ}_{Q,p}\B)}{\Vol(\Pi^\circ_{Q,p}\B)^{\frac{1}{m}}},\end{equation}
with equality if and only if $L=\Pi^{\circ}_{Q,p} E$ for any origin symmetric ellipsoid $E\in\conbodio[n].$ Note that $$\Gamma_{Q,p}\Pi^\circ_{Q,p}\B =  \left(\frac{m}{(nm+p)\omega_n}\right)^{\frac{1}{p}}\B$$ for all $Q\in\conbodo[m]$ (see \cite[Lemma 3.12]{HLPRY25_2}), but $\Pi^\circ_{Q,p}\B$ is not necessarily a dilate of $B^{nm}_2$ (e.g. when $m\geq 2,p=1$, and $Q$ is a simplex).

For $m\in \N$ and $Q\in\conbodo[m]$, we define the $m$th-order polar projection body $\Pi^\circ_{Q,p} f \subset \R^{nm}$ of a function $f$ in $W^{1,p}(\R^n)$ (or $BV(\R^n)$ if $p=1$) by the gauge $\|\theta\|_{\Pi^\circ_{Q,p} f} := \Lp$. This set contains the origin in its interior. Furthermore, we have the equality 
\begin{equation}
\label{eq:better_energy}
\Ep f = d_{n,p}(Q) (nm)^{-\frac{1}{nm}} \Vol(\Pi^\circ_{Q,p} f)^{-\frac{1}{nm}}. 
\end{equation}
Next, we require a regularity lemma for $ \Gamma_{Q,p}\Pi^\circ_{Q,p} f$. 

\begin{lemma}\label{l:smoothness} Let $p >1$, $m,n \in \N$, and let $f \in W_{0}^{1,p}(\R^n), f \neq 0$. Define the function $h(z) := h_{\Gamma_{Q,p}\Pi^\circ_{Q,p} f}(z)$.
Then:
\begin{itemize}
    \item[(i)] If $Q \in \conbodio[m]$, then $h$ belongs to the class $C^1(\R^n\setminus\{o\})$.
    \item[(ii)] If $Q\in\conbod[m]_{\operatorname{s}}$, then $h \in C^2(\R^n \setminus \{0\})$. 
\end{itemize} 
\end{lemma}
\begin{proof} We proceed first with item (i), that is, we begin by showing that our function is of class $C^1(\R^n\setminus\{o\})$. For ease, we set $M=\Pi^\circ_{Q,p} f$. We first use \eqref{eq:centroid_polar} to write, for $z\in \R^n,$
\begin{equation}
\label{e:map}
h(z)=\left(\frac{1}{nm+p}\frac{1}{\Vol(M)} \int_{\S} \Lp^{-nm-p} h_Q(\theta^tz)^p d\theta\right)^\frac{1}{p}, \quad z \in \R^n.
\end{equation}
It suffices to show that the map in \eqref{e:map} is $C^1$ outside the origin. Define the function
\[
g(y,\xi):=h_Q(y^t\xi)^p.
\]
From Lemma~\ref{l:kernel_smooth}, item (a), we deduce immediately that 
\[
G(\xi):=\int_M g(y,\xi) dy, \quad \xi \in \R^n
\]
is continuous in $\Rn$. Fix $\xi \neq o$ and set
\[
Q(\xi) = \int_M \nabla_\xi g(y,\xi) dy,
\]
which is well-defined by Lemma~\ref{l:kernel_smooth}, item $(c)$. For any sequence $\eta_j \to o$ in $\Rn$,
\begin{align*}
\frac{\big| G(\xi+\eta_j) - G(\xi) -\langle Q(\xi),\eta_j\rangle\big|}{|\eta_j|}
&\leq
\int_M
\frac{|g(y,\xi+\eta_j) - g(y,\xi) - \langle\nabla_\xi g(y,\xi),\eta_j\rangle|}{|\eta_j|}
dy.
\end{align*}
By Lemma~\ref{l:kernel_smooth}, item $(c)$, the integrand tends to $0$ for a.e. $y \in \Rnhi$. By Lemma~\ref{l:kernel_smooth}, items $(a)$, $(b)$ and $(d)$, the integrand is bounded, uniformly in $j$, by a constant multiple of $|y|^p$. Since $M$ is compact, we may apply the dominated convergence theorem to deduce that the last integral tends to $0$. Since this happens for every sequence $\eta_j$, we deduce that $G$ is differentiable at $\xi$ and its differential is $Q(\xi)$. 

Next, let $\xi_j \to \xi$. For each $j$, let $A(\xi_j)$ and $A(\xi)$ be the sets from Lemma~\ref{l:kernel_smooth}(c). The set
\[
A(\xi)\cap\bigcap_{j=1}^\infty A(\xi_j)
\]
has full measure in $\Rnhi$. For every $y$ in this set, the convex function $g(y,\cdot)$ is differentiable at $\xi$ and at every $\xi_j$. Hence
\[
\nabla_\xi g(y,\xi_j)\to \nabla_\xi g(y,\xi)
\]
as $j\to\infty$. We compute
\[
|Q(\xi_j) - Q(\xi)| \leq \int_M |\nabla_\xi g(y, \xi_j) - \nabla_\xi g(y, \xi)|d y.
\]
By Lemma~\ref{l:kernel_smooth}, items $(b)$ and $(d)$, the integrand is bounded, uniformly in $j$, by a constant multiple of $|y|^p$. Thus, we may apply the dominated convergence theorem and obtain that $Q(\xi_j) \to Q(\xi)$, showing that $G$ is $C^1$ on $\R^n\setminus\{o\}$. Finally, an application of the chain rule proves that the map in \eqref{e:map} is of class $C^1(\R^n\setminus\{o\})$.

We now prove item (ii). Assume that $Q\in \conbod[m]_{\operatorname{s}}$. Fix any $\xi \in \R^n \setminus \{0\}$. Define the set
\[
A_\xi = \{\theta \in \S \colon \theta^t \xi = o\}
\]
and note that $A_\xi$ is a set of measure zero in $\S$ (its dimension is strictly smaller than that of $\S$ since $\xi \neq o$). Assume that $\theta \in \S$ is such that $\theta^t\xi \neq o$. Then two direct computations show that, for each $j,k =1,\dots, n$,
\begin{align*}
\frac{\partial}{ \partial \xi_j}h_Q(\theta^t\xi)^p
=
ph_Q(\theta^t\xi)^{p-1}
\sum_{\ell=1}^m\partial_\ell h_Q(\theta^t\xi)\theta_{\ell,j}
\end{align*}
and
\begin{align*}
\frac{\partial^2}{ \partial \xi_j \partial \xi_k}h_Q(\theta^t\xi)^p
&=
p h_Q(\theta^t\xi)^{p-2}
\sum_{\ell,\omega=1}^m
\left(
(p-1)\partial_\ell h_Q(\theta^t\xi) \partial_\omega h_Q(\theta^t\xi)
\right.\\
&\hspace{5cm}\left.
+h_Q(\theta^t\xi)\partial_{\ell,\omega} h_Q(\theta^t\xi)
\right)
\theta_{\ell,j}\theta_{\omega,k}.
\end{align*}
Here $\partial_\ell$ and $\partial_{\ell\omega}$ denote the first and second order partial derivatives, and $\theta_{\ell,j}$ denotes the $j$th component of $\theta_\ell$. Each of the above functions is continuous by the assumption that $Q\in \conbod[m]_{\operatorname{s}}$.

We next justify differentiating underneath the integral sign. Since $h_Q$ is $1$-homogeneous and belongs to $C^2(\R^m\setminus\{o\})$, there are positive constants $c,C$ such that, for every $u\neq o_m$,
\[
c|u|\leq h_Q(u)\leq C|u|,
\qquad
|\nabla h_Q(u)|\leq C,
\qquad
|D^2h_Q(u)|\leq \frac{C}{|u|}.
\]
Consequently,
\[
\left|
\frac{\partial^2}{\partial \xi_j\partial \xi_k}
h_Q(\theta^t\xi)^p
\right|
\leq C|\theta^t\xi|^{p-2}.
\]
Since $p>1$, the function
\[
\theta\mapsto |\theta^t\xi|^{p-2}
\]
is integrable on $\S$ for every $\xi\neq o$, and these functions are locally uniformly integrable as $\xi$ varies in $\R^n\setminus\{o\}$. Moreover, $\Lp$ is bounded from below by a positive constant on $\S$. Thus, differentiation underneath the integral sign is justified twice.

Finally, since $Q\in \conbod[m]_{\operatorname{s}}$, we can differentiate underneath the integral sign and apply the Leibniz rule twice to obtain
\begin{align*}
&\frac{\partial^2}{\partial \xi_k \partial \xi_j} h(\xi)^p
=
\frac{1}{nm+p}\frac{1}{\Vol(M)}
\frac{\partial^2}{\partial \xi_k \partial \xi_j}
\left(
\int_{\S}
\Lp^{-nm-p} h_Q(\theta^t\xi)^p d\theta
\right)\\
&=
\frac{1}{nm+p}\frac{1}{\Vol(M)}
\frac{\partial^2}{\partial \xi_k \partial \xi_j}
\left(
\int_{\S \setminus A_\xi}
\Lp^{-nm-p} h_Q(\theta^t\xi)^p d\theta
\right)\\
&=
\frac{1}{nm+p}\frac{1}{\Vol(M)}
\int_{\S\setminus A_\xi}
\Lp^{-nm-p}
\frac{\partial^2}{\partial \xi_k \partial \xi_j}
[h_Q(\theta^t\xi)^p]
d\theta\\
&=
\frac{1}{nm+p}\frac{1}{\Vol(M)}
\int_{\S\setminus A_\xi}
\Lp^{-nm-p}
ph_Q(\theta^t\xi)^{p-2}\\
&\quad\quad\quad \times
\left(
\sum_{\ell,\omega=1}^m
\left(
(p-1)\partial_\ell h_Q(\theta^t\xi) \partial_\omega h_Q(\theta^t\xi)
+h_Q(\theta^t\xi)\partial_{\ell,\omega} h_Q(\theta^t\xi)
\right)
\theta_{\ell,j}\theta_{\omega,k}
\right)
d\theta.
\end{align*}
Thus $h^p$ belongs to $C^2(\R^n\setminus\{o\})$. Since $h(\xi)>0$ for $\xi\neq o$, an application of the chain rule shows that $h\in C^2(\R^n\setminus\{o\})$, as required.
\end{proof}

Next, we analyze the energy $\Ep$. The next lemma will be handy for our purposes; its proof is a standard modification of the proof of \cite[Theorem~2.1]{LM24_3}.

\begin{lemma}\label{l:lsc} Let $\Omega$ be a bounded, open set in $\R^n$, $Q\in \conbod[m]_{\operatorname{s}}$, and $p > 1$. Then the function $\Ep \colon W_0^{1,p}(\Omega) \to [0,\infty)$ is weakly lower semicontinuous, i.e., if $f_k \to f$ weakly in $W_0^{1,p}$, then 
\[
\Ep f \leq \liminf_{ k \to \infty} \Ep f_k. 
\]
\end{lemma}

We will make frequent use of the following modified version of \cite[Theorem~5.1]{LRZ24}. 

\begin{lemma}\label{l:useful} Let $\Omega \subset \R^n$ be a bounded, open set, $Q \in\conbodio[m]$ and $p \geq 1$. Then the following hold: 
\begin{itemize}
    \item[(i)] there is a constant $c_0:=c_0(Q,p,\Omega) > 0$ such that for any $\theta\in \S$ and any $f \in C^1(\R^n)$, with support contained in $\Omega$,
    \[
    \|h_Q(\theta^t \nabla f(\cdot))\|_{L^p(\R^n)} \geq c_0 \|f\|_{L^p(\Omega)};
    \]
    \item[(ii)] there is an explicit constant $d_0:=d_0(Q,p,\Omega) > 0$ such that for any $f \in C^1(\R^n)$, with support contained in $\Omega$,
    \[
    \Ep f \geq d_0 \|f\|_{L^p(\Omega)}^{\frac{nm-1}{nm}} \||\nabla f|\|_{L^p(\Omega)}^\frac{1}{nm}.
    \]
\end{itemize}
\end{lemma} 
The inequalities in Lemma~\ref{l:useful} extend to $W_0^{1,p}(\Omega)$ ($p>1$) and $\operatorname{BV}(\Omega)$ $(p=1)$ by approximation.

We note that items (i) and (ii) were originally shown for origin-symmetric convex bodies, but the general case follows. Indeed, the constants $c_0$ and $d_0$ were not sharp, and so we have no qualms about making them even smaller to allow non-symmetric $Q$ in the following way: observe that, for every $Q\in\conbodio[m]$ there exists a constant $c$, depending on $Q$, such that
\[
c (\conv (Q,-Q)) \subseteq Q \subseteq \conv (Q,-Q).
\]
Thus, we have the existence of a $\tilde c$ depending on $Q$ and $p$ such that
\[
\|h_Q(\theta^t \nabla f(\cdot))\|_{L^p(\R^n)} \geq \tilde c \|h_{\conv(Q,-Q)}(\theta^t  \nabla f(\cdot))\|_{L^p(\R^n)} \text{ and }  \mathcal{E}_{ \conv (Q,-Q),p} f\leq \tilde c \Ep f
\]
for every such function $f$. Finally, $\tilde c$ can be absorbed into $c_0$ and $d_0$, respectively, in items (i) and (ii) in the above lemma.

We require the following $m$th-order affine Sobolev inequality from \cite[Theorem~1.1]{HLPRY25_2} for $p=1$ and the $m$th-order affine P\'olya-Szeg\H{o} principle from \cite[Theorem~1.3]{LRZ24}. 

\begin{lemma}\label{l;Sobolevlemma} Let $Q\in\conbodo[m]$ and $p \geq 1$. Then, we have the following:

\begin{itemize}
    \item[(a)] (Affine Sobolev inequality): Assume that $f \in \operatorname{BV}(\R^n)$ is compactly supported. Then, $$\mathcal{E}_{Q,1} f \geq a_{n,1} \|f\|_{L^{\frac{n}{n-1}}(\R^n)},$$ where $a_{n,1}$ is a sharp constant, specifically the $L^1$ Aubin-Talenti Sobolev constant. There is equality if and only if $f = a \chi_E$ for some $a >0$ and some ellipsoid $E \subset \R^n$. 

    \item[(b)] (Affine P\'olya-Szeg\H{o} principle): Let $p >1$ and assume that $f \in W_0^{1,p}(\R^n)$. Then $$\Ep f \geq \Ep f^\star = \||\nabla f^\star|\|_{L^p(\R^n)}.$$ Furthermore, if $f$ is nonnegative and satisfies the consistency condition 
    \[
    \vol(\{x \colon |\nabla f^\star(x)|=0\} \cap \{x \colon 0 < f^\star(x) < \|f\|_\infty\}) =0,
    \]
    then there is equality in the P\'olya-Szeg\H{o} principle above if and only if $f(x) = F(|Ax|)$ for some $A \in GL(n)$ and $F \colon \R \to [0,\infty)$. 

\end{itemize}

\end{lemma}

\section{Analysis of the $m$th-order affine $p$-Laplacian} \label{s:PDEs}

Define the convex body 
\begin{equation}L_{Q,p,f} = \left(\frac{m}{(nm+p)\omega_n}\right)^{-\frac{1}{p}}\left(\frac{\vol[nm]\left(\Pi^\circ_{Q,p}B_2^n\right)}{\vol[nm]\left(\Pi^\circ_{Q,p} f\right)}\right)^\frac{1}{nm} \Gamma_{Q,p} \Pi^\circ_{Q,p} f.
\label{eq:L_body}
\end{equation}
The pre-factors are chosen so that, if $f$ has the property that $\Pi^\circ_{Q,p} f = \Pi^\circ_{Q,p}B_2^n$, then $L_{Q,p,f} = B_2^n$. The support function of $L_{Q,p,f}$ is precisely 
\begin{equation}
\label{eq:L_body_supp}
h_{L_{Q,p,f}}(z) = \left(\frac{m}{\omega_n}\right)^{-\frac{1}{p}}\frac{\vol[nm]\left(\Pi^\circ_{Q,p}B_2^n\right)^\frac{1}{nm}}{\vol[nm]\left(\Pi^\circ_{Q,p} f\right)^{\frac{1}{p}+\frac{1}{nm}}} \left(\int_{\S} \Lp^{-nm-p} h_Q(\theta^tz)^p d\theta\right)^\frac{1}{p}.
\end{equation}
It will be fruitful to use \eqref{eq:better_energy} to re-write the constant in \eqref{eq:L_body_supp}:
\begin{equation}
\label{eq:new_constant}
     \left(\frac{m}{\omega_n}\right)^{-\frac{1}{p}}\frac{\vol[nm]\left(\Pi^\circ_{Q,p}B_2^n\right)^\frac{1}{nm}}{\vol[nm]\left(\Pi^\circ_{Q,p} f\right)^{\frac{1}{p}+\frac{1}{nm}}} = d_{n,p}(Q)^{-\frac{nm}{p}}\Ep f^{\frac{nm}{p}+1}.
\end{equation}

Using the above definitions and identities, we prove the following invariance properties for the $m$th-order $p$-Laplace operator. Recall that this was defined in Definition~\ref{d:D} as the Wulff $p$-Laplacian with the choice $K= L_{Q,p,f}$.

\begin{proposition}\label{p:properties} Let $p >1$ and $Q\in \conbod[m]_{\operatorname{s}}$, and let $\Omega \subset \R^n$ be a bounded, open set.  Let $A \in \operatorname{SL}(n)$, $f \in W_0^{1,p}(\Omega)$, $a >0$ and $g(\cdot) = f(A\cdot)$.  Then we have the following:
\begin{itemize}
    \item[(a)] $\D g(x) = (\D f)(Ax)$;
    \item[(b)] $\D (a f)(x) = a^{p-1} \D f(x)$;
    \item[(c)] if $f$ is a radial function, then $\D f = \Delta_{p} f$ in $\R^n$. 
\end{itemize}
\end{proposition}

\begin{proof} It follows from the definition \eqref{eq:L_body} of $L_{Q,p,f}$ and \cite[Proposition~3.3]{HLPRY25_2} that one has relationship 
\[
\Pi^\circ_{Q,p}g = \overline{A^{-1}} \Pi^\circ_{Q,p}f.
\]
Therefore, 
\begin{align*}
L_{Q,p,g} &= \left(\frac{m}{(nm+p)\omega_n}\right)^{-\frac{1}{p}}\left(\frac{\vol[nm]\left(\Pi^\circ_{Q,p}B_2^n\right)}{\vol[nm]\left(\Pi^\circ_{Q,p} g\right)}\right)^\frac{1}{nm} \Gamma_{Q,p} \Pi^\circ_{Q,p} g
\\
&=\left(\frac{m}{(nm+p)\omega_n}\right)^{-\frac{1}{p}}\left(\frac{\vol[nm]\left(\Pi^\circ_{Q,p}B_2^n\right)}{\vol[nm]\left(\overline{A^{-1}}\Pi^\circ_{Q,p} f\right)}\right)^\frac{1}{nm} \Gamma_{Q,p} \left(\overline{A^{-1}}\Pi^\circ_{Q,p} f\right)
\\
&=A^{-1}\left[\left(\frac{m}{(nm+p)\omega_n}\right)^{-\frac{1}{p}}\left(\frac{\vol[nm]\left(\Pi^\circ_{Q,p}B_2^n\right)}{\vol[nm]\left(\Pi^\circ_{Q,p} f\right)}\right)^\frac{1}{nm} \Gamma_{Q,p} \Pi^\circ_{Q,p} f\right]
\\
&=A^{-1} L_{Q,p,f}.
\end{align*}
By \cite[Proposition~2(i)]{HJM21}, we can finally write
\[
\D g(x) = \Delta_{p,L_{Q,p,g}} g(x) = \Delta_{p, A^{-1}L_{Q,p,f}} f(Ax) = \D f(Ax),
\]
which establishes item (a).

For item (b), first observe that $\Ep (af) = a \Ep f$. Then, by \eqref{eq:L_body_supp} and \eqref{eq:new_constant}, we have the point-wise identity
\begin{align*}
h_{L_{Q,p,(af)}}(z) &= d_{n,p}(Q)^{-\frac{nm}{p}}\Ep (af)^{\frac{nm}{p}+1} \left(\int_{\S} \|h_Q(\theta^t \nabla (af)(\cdot))\|_{L^p(\R^n)}^{-nm-p} h_Q(\theta^tz)^p d\theta\right)^\frac{1}{p}
\\
& = a^{\frac{nm}{p}+1} d_{n,p}(Q)^{-\frac{nm}{p}}\Ep f^{\frac{nm}{p}+1} \left(\int_{\S} \|h_Q(\theta^t \nabla (af)(\cdot))\|_{L^p(\R^n)}^{-nm-p} h_Q(\theta^tz)^p d\theta\right)^\frac{1}{p}
\\
& =  d_{n,p}(Q)^{-\frac{nm}{p}}\Ep f^{\frac{nm}{p}+1} \left(\int_{\S} \|h_Q(\theta^t \nabla f(\cdot))\|_{L^p(\R^n)}^{-nm-p} h_Q(\theta^tz)^p d\theta\right)^\frac{1}{p} = h_{L_{Q,p,f}}(z),
\end{align*}
showing that $L_{Q,p,(af)} = L_{Q,p,f}$ for all $a>0$. Therefore, we have
\begin{align*}
\D (af)(x) &= \Delta_{p,L_{Q,p,(af)}} (af)(x) = \Delta_{p, L_{Q,p,f}} (af)(x) 
\\
&= -\operatorname{div}(h_{L_{Q,p,f}}^{p-1}(\nabla (af)) \nabla h_{L_{Q,p,f}}(\nabla (af))) 
\\
&= -a^{p-1} \operatorname{div}(h_{L_{Q,p,f}}^{p-1}(\nabla f) \nabla h_{L_{Q,p,f}}(\nabla f)) = a^{p-1}\D f(x).
\end{align*}

For item (c), it was shown in \cite{HJM21} that, if $K=B_2^n$ and $f$ is radial, then $\Delta_{p,B_2^n} f = \Delta_p f$. Therefore, it suffices to show that, if $f$ is radial, then $L_{Q,p,f}=B_2^n$. But this is immediate from the fact that, if $f$ is radial, then $\Pi^\circ_{Q, p} f = \Pi^\circ_{Q, p} r B_2^n$ for some $r>0$. 
\end{proof}

The next proposition concerns robust regularity properties enjoyed by the operator $\D$ acting on $W_0^{1,p}(\Omega)$; its proof follows along the lines of the proof of \cite[Proposition~4]{HJM21}.

\begin{proposition}\label{p:regularity1} Let $\Omega \subset \R^n$ be a bounded, open set, $Q\in \conbod[m]_{\operatorname{s}}$, and $p > 1$. Suppose that $f \in W_0^{1,p}(\Omega)$ is a solution to the problem 
\[
\begin{cases}
\D f = h_0 &\text{in } \Omega,\\
f = 0 &\text{on } \partial \Omega,
\end{cases}
\]
where $h_0 \colon \Omega \to \R$ is a measurable function. Then the following facts are valid:
\begin{enumerate}
    \item[(a)] if $p <n$ and $h_0 \in L^{\frac{n}{p}}(\Omega)$, then $f \in L^k(\Omega)$ for every $k \geq 1$;
    \item[(b)] if $p \leq n$ and $h_0 \in L^q(\Omega)$ for some $q > \frac{n}{p}$, then $f \in L^\infty(\Omega)$;
    \item[(c)] if $h_0 \in L^\infty(\Omega)$, then $f \in C^{1,\alpha}(\Omega)$ for some $0 < \alpha < 1$;
    \item[(d)] if $h_0 \in L^\infty(\Omega)$, then $f \in C^{1,\alpha}(\overline{\Omega})$ for some $\alpha \in (0,1)$ provided the boundary of $\Omega$ is of class $C^{2,\alpha}$. 
\end{enumerate}
\end{proposition}

\begin{proof} 
We may assume that $f$ is non-zero. Consider the vector field $\nu \colon \R^n \to \R^n$ defined by 
\[
\nu(z) = d_{n,p}(Q)^{nm}(\Ep f)^{-nm-p}h_{L_{Q,p,f}}(z)^{p-1}\nabla h_{L_{Q,p,f}}(z). 
\]
We also require the related differential operator $\mathcal{L}$ given by the equation
\[
\mathcal{L} g = - \text{div}(\nu(\nabla g)), \quad g \in W_{0}^{1,p}(\Omega). 
\]
Since $ \mathcal{L} f = d \D f$ for some constant $d >0$, and since $f$ is non-zero, in view of arguments based on the works \cite{DeGiorgi, NASH, MJ60,DE83,Tolksdorf} (see \cite{HJM21} for more details), in order to prove items (a)-(d), it is enough to show that our operator fits into the situation of Lemma~\ref{l:Tolksdorf}. In fact, we will show that $\mathcal{L}$ is a Wulff-type degenerate quasilinear elliptic operator.

First, observe that by employing the Cauchy-Schwarz inequality followed by Lemma~\ref{l:useful}(i), there is a constant $c_0 > 0$ such that 
\begin{equation}\label{e:ineqs}
c_0 \|f\|_{L^p(\Omega)} \leq \Lp \leq \||\nabla f|\|_{L^p(\Omega)} \quad \text{ for all } \theta \in \S.
\end{equation}

By Lemma~\ref{l:smoothness}, we know that the function $h_{L_{Q,p,f}}$ belongs to the class $C^2(\R^n\setminus\{o\})$. 

We will begin by working on verifying the first condition of Lemma~\ref{l:Tolksdorf}. Since the following computations are fairly long, we adopt the notation $\|\cdot\| = \Lp$ for every $\theta \in \S$. For $z \in \R^n$ and any $j =1,\dots,n$, a direct computation shows that the $j$th component of $\nu$ is given by
\begin{align*}
\nu_j(z) &= d_{n,p}^{nm}(Q) (\Ep f)^{-nm-p} \frac{\partial}{\partial z_j}\left( \frac{1}{p} h_{L_{Q,p,f}}^p(z)\right)\\
&=\int_{\S} \|\theta\|^{-nm-p} h_Q(\theta^t z)^{p-1} \sum_{i=1}^m \partial_ih_Q(\theta^t z) \theta_{i,j} d\theta,
\end{align*}
where $\theta_{i,j}$ denotes the $j$th component of $\theta_i \in \R^n$.   

Next, we compute 
\begin{align*}
&\frac{\partial \nu_j(z)}{\partial z_k} = \frac{\partial}{\partial z_k}\left( \int_{\S} \|\theta\|^{-nm-p} h_Q(\theta^t z)^{p-1} \sum_{i=1}^m \partial_i h_Q(\theta^t z) \theta_{i,j} d\theta\right)\\
&= (p-1) \int_{\S} \|\theta\|^{-nm-p} h_Q(\theta^t z)^{p-2} \sum_{i,\ell=1}^m \partial_i h_Q(\theta^tz) \partial_\ell h_Q(\theta^t z) \theta_{i,j} \theta_{\ell,k} d\theta\\
&\quad+ \int_{\S} \|\theta\|^{-nm-p} h_Q(\theta^tz)^{p-1} \sum_{i,\eta =1}^m \partial_{i\eta}h_Q(\theta^tz) \theta_{i,j}\theta_{\eta,k} d\theta. 
\end{align*}
Consequently, for every $z \in \R^n \setminus \{0\}$ and $x \in \R^n$, we obtain 
\begin{align*}
\sum_{j,k=1}^n \frac{\partial \nu_j(z)}{\partial z_k} x_jx_k
&= (p-1) \int_{\S} \|\theta\|^{-nm-p} h_Q(\theta^t z)^{p-2} 
\left(\sum_{i=1}^m \partial_i h_Q(\theta^t z) \langle x,\theta_i\rangle \right)^2 d\theta\\
&\quad+ \int_{\S} \|\theta\|^{-nm-p} h_Q(\theta^t z)^{p-1}
\sum_{i,\eta =1}^m \partial_{i\eta} h_Q(\theta^tz) 
\langle x,\theta_i\rangle \langle x, \theta_\eta\rangle d\theta. 
\end{align*}
From the convexity of $Q$, the function $h_Q$ is convex, and therefore the matrix 
$(\partial_{i\eta} h_Q(\cdot))_{i,\eta =1,\dots,m}$ is positive semi-definite. Consequently, the second term in the above computation is nonnegative. Therefore, the right-hand inequality of \eqref{e:ineqs} implies 
\begin{align*}
\sum_{j,k=1}^n \frac{\partial \nu_j(z)}{\partial z_k} x_jx_k
&\geq (p-1) \||\nabla f|\|_{L^p(\Omega)}^{-nm-p} 
\int_{\S}h_Q(\theta^t z)^{p-2} 
\left(\sum_{i=1}^m \partial_i h_Q(\theta^t z) \langle x,\theta_i\rangle \right)^2 d\theta\\
&\geq c_1 \||\nabla f|\|_{L^p(\Omega)}^{-nm-p} |z|^{p-2}|x|^2,
\end{align*}
where 
\begin{equation}
\label{eq:c_1_def}
c_1 := (p-1)
\min_{z,x \in \s}
\left[
\int_{\S}
h_Q(\theta^tz)^{p-2}
\left(
\sum_{i=1}^m
\partial_i h_Q(\theta^tz)\langle x,\theta_i\rangle
\right)^2
d\theta
\right].
\end{equation}
We claim that $c_1$ is finite and positive. Indeed, since $Q$ contains the origin in its interior, there are constants $r_Q,R_Q>0$ such that
\begin{equation}
\label{eq:h_Q_bound}
r_Q|u|\leq h_Q(u)\leq R_Q|u|,
\qquad u\in\R^m.
\end{equation}
Moreover, the $1$-homogeneity of $h_Q$ implies that $\nabla h_Q$ is $0$-homogeneous. Thus, the only possible singularity in the integrand is of order $|\theta^tz|^{p-2}$, which is integrable since $p>1$. Hence, the integral is finite and depends continuously on $(z,x)\in\s\times\s$.

For fixed $z,x\in\s$, the integrand in \eqref{eq:c_1_def} is nonnegative and is not identically zero. Indeed, choose $u\in\R^m\setminus\{o\}$. By Euler's formula \eqref{eq:support_homo}, one may choose $\theta\in\S$ with $\theta^tz=u$ so that
\[
\sum_{i=1}^m \partial_i h_Q(u)\langle x,\theta_i\rangle\neq0.
\]
By continuity, the integrand is then positive on a set of positive measure. Thus the integral is positive for every $(z,x)\in\s\times\s$, and compactness yields $c_1>0$. Therefore, we have verified the first condition of Lemma~\ref{l:Tolksdorf}.

We proceed with verifying that we satisfy the second condition of Lemma~\ref{l:Tolksdorf}. To start, define 
\begin{align*}
c_2 &:= (p-1)n^2 \int_{\S} \int_{O(n)} h_Q(\theta_T)^{p-2} 
\sum_{i,\ell=1}^m |\partial_i h_Q(\theta_T)| |\partial_\ell h_Q(\theta_T)| dT d\theta\\
&\quad+n^2 \int_{\S} \int_{O(n)} h_Q(\theta_T)^{p-1} 
\sum_{i,\eta=1}^m |\partial_{i\eta} h_Q(\theta_T)| dT d\theta,
\end{align*}
and $\theta_T := ((T^* \theta_1)_1,\dots, (T^*\theta_m)_1)$. Then, the $1$-homogeneity of $h_Q$ implies that $\nabla h_Q$ is $0$-homogeneous and $D^2h_Q$ is $(-1)$-homogeneous. Together with \eqref{eq:h_Q_bound}, both integrands defining $c_2$ are therefore bounded, up to a constant, by $|\theta_T|^{p-2}.$ This is integrable for $p>1$. Hence $c_2$ is finite, and it is positive.

With this in hand, the lower bound of \eqref{e:ineqs} yields
\begin{align*}
\sum_{j,k=1}^n \left|\frac{\partial \nu_j(z)}{\partial z_k}\right| 
&\leq (p-1) \sum_{j,k=1}^n \int_{\S} \|\theta\|^{-nm-p} h_Q(\theta^t z)^{p-2} 
\sum_{i,\ell=1}^m |\partial_i h_Q(\theta^tz) \partial_\ell h_Q(\theta^tz)|
|\theta_{i,j}||\theta_{\ell,k}|d\theta\\
&\quad+ \sum_{j,k=1}^n  \int_{\S} \|\theta\|^{-nm-p}h_Q(\theta^tz)^{p-1} 
\sum_{i,\eta=1}^m|\partial_{i\eta} h_Q(\theta^tz)|
|\theta_{i,j}||\theta_{\eta,k}|d\theta\\
&\leq n^2(p-1)(c_0 \|f\|_{L^p(\Omega)})^{-nm-p}
\int_{\S} h_Q(\theta^tz)^{p-2}  
\sum_{i,\ell=1}^m |\partial_i h_Q(\theta^tz)| |\partial_\ell h_Q(\theta^tz)|d\theta\\
&\quad+ n^2 (c_0 \|f\|_{L^p(\Omega)})^{-nm-p} 
\int_{\S} h_Q(\theta^tz)^{p-1} 
\sum_{i,\eta=1}^m |\partial_{i\eta}h_Q(\theta^tz)| d\theta\\
&\leq c_2 (c_0 \|f\|_{L^p(\Omega)})^{-nm-p}|z|^{p-2},
\end{align*}
We have verified the hypotheses of Lemma~\ref{l:Tolksdorf}, and completed the proof.
\end{proof}

\section{Rayleigh quotient and minimization problem}
In this section, we study $\L$. However, we do so more generally. Throughout this section, we fix a bounded $\Omega\subset \R^n$, $1\leq p<\infty$, $1\leq q\leq p$ and $Q\in\conbod[m]_{\operatorname{s}}$. Then, we define the $q$th affine $m$th-order Rayleigh quotient $\Rqq f$ to be
\[
\Rqq f = \frac{\Ep f}{\|f\|_{L^q(\Omega)}}.
\]
Observe that 
\[\mathcal{R}^\mathcal{A}_{Q,p,p}f=(\Rq f)^\frac{1}{p}.\]
We then set, for a bounded, open set $\Omega,$
\begin{equation}\label{e:poincareconst_q}
\Lq = \inf\left\{\Rqq f \colon f \in W^{1,p}_0(\Omega) \setminus \{0\} \right\}.
\end{equation}
The relation between $\L$ and $\Lq$ is precisely 
\[
\lambda_{1,p}^{\mathcal{A},p}(Q,\Omega)  = (\L)^\frac{1}{p}.
\]
Recall the classical (Euclidean) minimization problem is 
\begin{equation}
\lambda_{1,p}^q(\Omega) = \inf\left\{\mathcal{R}_{p,q} f \colon f \in W^{1,p}_0(\Omega) \setminus \{0\} \right\}, \text{where} \quad \mathcal{R}_{p,q}f = \frac{\||\nabla f |\|_{L^p(\Omega)}}{\|f\|_{L^q(\Omega)}}.
\label{eq:classical_min}
\end{equation}
Of course, $W^{1,p}_0(\Omega)$ should be $\operatorname{BV}(\Omega)$ if $q=p=1$ in \eqref{e:poincareconst_q} and \eqref{eq:classical_min}.
The quantity $\lambda_{1,p}^{\mathcal{A},q}(Q,\Omega)$ can be directly compared to $\lambda_{1,p}^q(\Omega)$: by \cite[Theorem 1.2]{LRZ24} and Lemma~\ref{l:useful}(ii), and then H\"older's inequality, we have
\begin{align*}
     \||\nabla f|\|_{L^p(\Omega)} &\geq \Ep f \geq d_0(Q,p,\Omega) \|f\|_{L^p(\Omega)}^{\frac{nm-1}{nm}} \||\nabla f|\|_{L^p(\Omega)}^\frac{1}{nm}
     \\
     & \geq d_0(Q,p,q,\Omega)\|f\|_{L^q(\Omega)}^{\frac{nm-1}{nm}} \||\nabla f|\|_{L^p(\Omega)}^\frac{1}{nm}
     \\
     & \geq d_0(Q,p,q,\Omega)\|f\|_{L^q(\Omega)} \left(\frac{\||\nabla f|\|_{L^p(\Omega)}}{\|f\|_{L^q(\Omega)}}\right)^\frac{1}{nm}.
\end{align*}
This yields the inequality
\[
\mathcal{R}_{p,q} f \geq \Rqq f \geq d_0(Q,p,q,\Omega)\left(\mathcal{R}_{p,q} f\right)^\frac{1}{nm}.
\]
Taking infima, we must have
\begin{equation}
\label{eq:relations}
    \lambda_{1,p}^q(\Omega) \geq \lambda_{1,p}^{\mathcal{A},q}(Q,\Omega) \geq d_0(Q,p,q,\Omega)\left(\lambda_{1,p}^q(\Omega)\right)^\frac{1}{nm}.
\end{equation}

\subsection{Minimization Problem}
This section concerns studying the functions attaining the infima appearing in \eqref{e:poincareconst} and \eqref{e:poincareconst_q}. We study the following equation on $W^{1,p}_0(\Omega)$:
 \begin{equation}\label{e:weak_q}
\D f = (\Lq)^p \|f\|_{L^q(\Omega)}^{p-q}|f|^{q-2} f \quad \text{ in } \Omega,
\end{equation}
which reduces to \eqref{e:weak} when $p=q$. Note that, if $q=p=1$, then  $W^{1,p}_0(\Omega)$ should be $\operatorname{BV}(\Omega)$.

We will follow very closely the methods of \cite{HJM21,HX22}. We first study the associated Euler-Lagrange equations. 

\begin{proposition}\label{p:eleq_q} 
  Let $\Omega \subset \R^n$ be a bounded, open set, $Q\in\conbod[m]_{\operatorname{s}}$, and let $1< q\leq p<\infty$. Suppose there exists a function $f \in W_0^{1,p}(\Omega)$ minimizing $\Lq$. Then, $f$ satisfies 
\[
\int_\Omega \langle h_{L_{Q,p,f}}(\nabla f)^{p-1} \nabla h_{L_{Q,p,f}}(\nabla f), \nabla \varphi \rangle dx = (\Lq)^p \|f\|_{L^q(\Omega)}^{p-q} \int_\Omega |f|^{q-2} f\varphi dx
\]
for every $\varphi \in W^{1,p}_0(\Omega)$. In other words, $f$ is a weak solution to the Dirichlet-type problem 
\[
\begin{cases}
\D f = (\Lq)^p \|f\|_{L^q(\Omega)}^{p-q}|f|^{q-2} f, &\text{in } \Omega, \\
f= 0, &\text{on } \partial \Omega.
\end{cases}
\]
\end{proposition}

\begin{proof} 
Suppose that $f \in W^{1,p}_0(\Omega)$ is a minimizer associated to $\Lq$. Then, for every $g \in W_0^{1,p}(\Omega)$, it holds $\Lq \leq \Ep g \|g\|_{L^q(\Omega)}^{-1}$, or equivalently, 
\begin{equation}\label{e:minima}
\int_{\S} \|h_Q(\theta^t \nabla g(\cdot))\|_{L^p(\R^n)}^{-nm} d\theta - d_{n,p}(Q)^{nm} (\Lq)^{-nm} \|g\|_{L^q(\Omega)}^{-nm} \leq 0,
\end{equation}
with equality when $g=f$. 

For a test function $\varphi \in W_0^{1,p}(\Omega)$ and $\e>0$, set $f_\e = f + \e \varphi$. We will differentiate the left-hand side of \eqref{e:minima} with $g = f_\e$ at $\e = 0$. For brevity, we write $\frac{\partial}{\partial \varphi} = \frac{d}{d\varepsilon}\big|_{\varepsilon=0}$. In what follows, $\partial_i h_Q$ denotes the $i$th partial derivative of $h_Q$ for $i=1,\dots, m$. First, we compute 
\begin{equation}\label{e:el1}
\begin{split}
&\frac{\partial}{\partial \varphi}\left[\int_{\S} \|h_Q(\theta^t \nabla f_\e(\cdot))\|_{L^p(\R^n)}^{-nm} d\theta\right]\\
&= - nm \int_{\R^n} \!\!\!\left\langle \nabla \varphi(x), \int_{\S}\!\!\!\! \|h_Q(\theta^t \nabla f(\cdot))\|_{L^p(\R^n)}^{-nm-p}h_Q(\theta^t \nabla f(x))^{p-1} \sum_{i=1}^m \partial_i h_Q(\theta^t \nabla f(x)) \theta_i d\theta \right\rangle dx,
\end{split}
\end{equation}
where in the last step we have applied Fubini's theorem.  Next, consider the following dilate of $L_{Q,p,f}$, the convex body 
\begin{equation}
\label{eq:K_body}
    K_{Q,p,f} =(nm+p)^\frac{1}{p}\vol[nm](\Pi^\circ_{Q,p} f)^{\frac{1}{p}} \Gamma_{Q,p} \Pi^\circ_{Q,p} f,
\end{equation}
with support function 
\begin{equation}
\label{eq:K_body_support}
h_{K_{Q,p,f}}(z) = \left(\int_{\S}  \|h_Q(\theta^t \nabla f(\cdot))\|_{L^p(\R^n)}^{-nm-p}h_Q(\theta^t z)^p d\theta\right)^\frac{1}{p},
\end{equation}
and observe that
\[
\nabla \left(\frac{1}{p} h_{K_{Q,p,f}}^p\right)(z) = \int_{\S} \|h_Q(\theta^t \nabla f(\cdot))\|_{L^p(\R^n)}^{-nm-p}h_Q(\theta^t z)^{p-1} \sum_{i=1}^m \partial_i h_Q(\theta^t z) \theta_i d\theta. 
\]
Thus, we can rewrite equation \eqref{e:el1} in the form 
\begin{equation}\label{e:el2}
\frac{\partial}{\partial \varphi}\left[\int_{\S} \|h_Q(\theta^t \nabla f_\e(\cdot))\|_{L^p(\R^n)}^{-nm} d\theta\right] = - nm \int_{\R^n} \left\langle \nabla  \varphi(x), \nabla \left(\frac{1}{p} h_{K_{Q,p,f}}^p\right)(\nabla f(x)) \right\rangle dx
\end{equation}
Similarly, it is readily verified that 
\begin{equation}\label{e:el3}
\frac{\partial}{\partial \varphi} [\|f_\e\|_{L^q(\Omega)}^{-nm}]  = -nm \|f\|_{L^q(\Omega)}^{-nm-q} \int_{\R^n}\varphi(x) |f(x)|^{q-2}f(x) dx.
\end{equation}
Combining equation \eqref{e:el1}, \eqref{e:el2} and \eqref{e:el3}, we obtain the following  weak formulation of \eqref{e:weak_q}: for every $\varphi \in W_0^{1,p}(\Omega)$, it holds that 
\begin{equation}\label{e:el4}
\int_{\R^n} \!\!\left\langle \nabla \varphi, \nabla \left(\frac{1}{p} h_{K_{Q,p,f}}^p\right)(\nabla f) \right\rangle dx = d^{nm}_{n,p}(Q) (\Lq)^{-nm} \|f\|_{L^q(\Omega)}^{-nm-q} \!\!\int_{\R^n}\varphi(x) |f|^{q-2}fdx .
\end{equation}

We now replace $K_{Q,p,f}$ with $L_{Q,p,f}$. Observe that 
\[
L_{Q,p,f}=d_{n,p}(Q)^{-\frac{nm}{p}}\left(\Ep f\right)^{\frac{nm}{p}+1}K_{Q,p,f}.
\]
Consequently, by the $p$-homogeneity of $h_K^p$,
\[
\nabla\left(\frac{1}{p}h_{L_{Q,p,f}}^p\right)=d_{n,p}(Q)^{-nm}\left(\Ep f\right)^{nm+p}\nabla\left(\frac{1}{p}h_{K_{Q,p,f}}^p\right).
\]
Multiplying \eqref{e:el4} by $d_{n,p}(Q)^{-nm}\left(\Ep f\right)^{nm+p}$ and using
\[
\Lq\|f\|_{L^q(\Omega)}=\Ep f,
\]
we obtain, for every $\varphi\in W_0^{1,p}(\Omega)$,
\begin{align*}
\int_{\R^n}\left\langle\nabla\varphi,\nabla\left(\frac{1}{p}h_{L_{Q,p,f}}^p\right)(\nabla f)\right\rangle dx
 =
(\Lq)^p\|f\|_{L^q(\Omega)}^{p-q}
\int_{\R^n}\varphi(x)|f(x)|^{q-2}f(x)\,dx.
\end{align*}
Since
\[
\nabla\left(\frac{1}{p}h_{L_{Q,p,f}}^p\right)= h_{L_{Q,p,f}}^{p-1}\nabla h_{L_{Q,p,f}},
\]
this is precisely the desired weak formulation of \eqref{e:weak_q}. This completes the proof. 
\end{proof}

Next, we show the \textit{existence} of functions obtaining the infima appearing in \eqref{e:poincareconst} and \eqref{e:poincareconst_q}. We handle them separately because, when $q\neq p$, we require $1<p<n$ due to the use of a Sobolev inequality. However, if $q=p$, then we may allow all $p\geq 1$.

\begin{theorem}\label{t:existence}
 Let $\Omega \subset \R^n$ be a bounded, open set, $Q\in \conbod[m]_{\operatorname{s}}$, and $p \geq 1$. Then:
 \begin{enumerate}
     \item the infimum defining $\L$ is attained for some $f_p \in W_0^{1,p}(\Omega)$ for $p >1$ and for some $f_1 \in \operatorname{BV}(\Omega)$ for $p =1$.
     \item if $1<p<n$, then, for every $q\in [1,p]$, the infimum defining $\Lq$ is attained for some $f_{p,q} \in W_0^{1,p}(\Omega)$. 
 \end{enumerate}
\end{theorem}

\begin{proof}
Consider first $p>1$. Let $(f_i)\subset W_0^{1,p}(\Omega)$ be a sequence such that $\|f_i\|_{L^q(\Omega)}=1$ for all $i$ and $\Ep f_i\to\Lq$ as $i\to\infty$.

We first specialize to the first claim, in which case $q=p$, and our sequence satisfies $\|f_i\|_{L^p(\Omega)}=1$ for all $i$ and $(\Ep f_i)^p\to\L$ as $i\to\infty$. By Lemma~\ref{l:useful}(ii), the sequence $(f_i)$ is bounded in $W_0^{1,p}(\Omega)$. By reflexivity and the Rellich-Kondrachov embedding theorem, after passing to a subsequence if necessary, there exists $f\in W_0^{1,p}(\Omega)$ such that $f_i\to f$ weakly in $W_0^{1,p}(\Omega)$ and
\[
\|f_i-f\|_{L^p(\Omega)}\to0.
\]
In particular, $\|f\|_{L^p(\Omega)}=1.$ By Lemma~\ref{l:lsc},
\[
\Ep f\leq\liminf_{i\to\infty}\Ep f_i.
\]
Consequently,
\[
\Rq f
=
(\Ep f)^p
\leq
\liminf_{i\to\infty}(\Ep f_i)^p
=
\L.
\]
Since, by definition, $\L\leq\Rq f$, we conclude that $\Rq f=\L$, proving the first claim when $p>1$.

We now prove the first claim when $p=1$. Let $(f_i)\subset \operatorname{BV}(\Omega)$ be a minimizing sequence normalized so that $\|f_i\|_{L^1(\Omega)}=1$ for every $i$. Thus,
\[
\Ep f_i\to\L.
\]
By Lemma~\ref{l:useful}(ii), the sequence $(f_i)$ is bounded in $\operatorname{BV}(\Omega)$. Hence, by the compactness theorem for functions of bounded variation \cite[Theorem 3.23]{AFP}, after passing to a subsequence if necessary, there exists $f\in\operatorname{BV}(\Omega)$ such that
\[
f_i\to f
\qquad\text{in }L^1(\Omega),
\]
and therefore $\|f\|_{L^1(\Omega)}=1$.

For $\theta\in\S$, recall that we used the convention
\[
\|h_Q (\theta^t\nabla f(\cdot))\|_{L^1(\R^n)}
=
\int_{\R^n}h_Q(\theta^t\eta_f(x))\,d|Df|(x).
\]
Since $z\mapsto h_Q(\theta^tz)$ is convex, positively $1$-homogeneous, and lower semicontinuous, we have
\[
\|h_Q (\theta^t\nabla f(\cdot))\|_{L^1(\R^n)}
\leq
\liminf_{i\to\infty}\|h_Q (\theta^t\nabla f_i(\cdot))\|_{L^1(\R^n)}.
\]
Consequently,
\[
\|h_Q (\theta^t\nabla f(\cdot))\|_{L^1(\R^n)}^{-nm}
\geq
\limsup_{i\to\infty}\|h_Q (\theta^t\nabla f_i(\cdot))\|_{L^1(\R^n)}^{-nm}.
\]
Moreover, by Lemma~\ref{l:useful}(i),
\[
\|h_Q (\theta^t\nabla f_i(\cdot))\|_{L^1(\R^n)}
\geq
c_0\|f_i\|_{L^1(\Omega)}
=
c_0,
\]
uniformly in $i$ and $\theta$. Thus, by the reverse Fatou inequality,
\[
\int_{\S}\|h_Q (\theta^t\nabla f(\cdot))\|_{L^1(\R^n)}^{-nm}\,d\theta
\geq
\limsup_{i\to\infty}
\int_{\S}\|h_Q (\theta^t\nabla f_i(\cdot))\|_{L^1(\R^n)}^{-nm}\,d\theta.
\]
It follows from the definition of $\mathcal E_{Q,1}$ that
\[
\mathcal E_{Q,1}f
\leq
\liminf_{i\to\infty}\mathcal E_{Q,1}f_i
=
\L.
\]
Since $\|f\|_{L^1(\Omega)}=1$, we conclude that
\[
\mathcal R^\mathcal A_{Q,1}f\leq\L.
\]
The reverse inequality follows from the definition of $\L$, and hence $f$ is a minimizer.

We now move on to the second claim. We first show that $\Ep f_i$ is uniformly bounded from below. Indeed, by applying Jensen's inequality, we have
\[
\|f_i\|_{L^p(\Omega)}
\geq
c\|f_i\|_{L^q(\Omega)}
=
c
\]
for some $c=c(\Omega,n,p,q)>0$. Also, there exists \cite[Theorem 3 on pg.~265]{Evans} a constant $k(\Omega,n,p,q)$ such that
\[
1
=
\|f_i\|_{L^q(\Omega)}
\leq
k(\Omega,n,p,q)\||\nabla f_i|\|_{L^p(\Omega)}.
\]
Therefore, Lemma~\ref{l:useful}(ii) gives
\[
\Rqq f_i
=
\Ep f_i
\geq
d_0\,
c^{1-\frac{1}{nm}}\,
k(\Omega,n,p,q)^{-\frac{1}{nm}}.
\]

From the same lemma, we have
\[
\||\nabla f_i|\|_{L^p(\Omega)}
\leq
\left(
\frac{\Ep f_i}
{d_0c^{1-\frac{1}{nm}}}
\right)^{nm}.
\]
For $i$ sufficiently large, we may take
\[
\Ep f_i\leq\Lq+1.
\]
Consequently, $(f_i)$ is bounded in $W_0^{1,p}(\Omega)$, and hence, after passing to a subsequence if necessary, there exists $f\in W_0^{1,p}(\Omega)$ such that
\[
f_i\to f
\quad\text{weakly in }W_0^{1,p}(\Omega)
 \quad \text{and} \quad
f_i\to f
\quad\text{in }L^q(\Omega).
\]
In particular, $\|f\|_{L^q(\Omega)}=1,$ so $f\neq0$, and by definition
\[
\Lq\leq\Rqq f.
\]
On the other hand, Lemma~\ref{l:lsc} gives
\[
\Ep f
\leq
\liminf_{i\to\infty}\Ep f_i.
\]
Since $\|f\|_{L^q(\Omega)}=1$, it follows that
\[
\Rqq f
=
\Ep f
\leq
\liminf_{i\to\infty}\Ep f_i
=
\Lq.
\]
Therefore,
\[
\Rqq f=\Lq,
\]
which proves the second claim.
\end{proof}

\subsection{The main regularity theorem}
\label{sec:regality}
 In this subsection, we prove Theorem~\ref{t:operatoranalysis}. The proof is rather involved and requires several preliminary results.

\begin{proposition}
\label{p:eleq}
Let $\Omega \subset \R^n$ be a bounded, open set, $Q\in \conbod[m]_{\operatorname{s}}$, and $p>1$. Let $f \in W_0^{1,p}(\Omega)\setminus\{0\}$. Then $f$ is a minimizer associated to $\L$ if and only if it satisfies
\begin{equation}
\label{eq:weak_PDE_solved}
\int_\Omega
\left\langle
h_{L_{Q,p,f}}(\nabla f)^{p-1}
\nabla h_{L_{Q,p,f}}(\nabla f),
\nabla \varphi
\right\rangle dx
=
\L
\int_\Omega |f|^{p-2}f\varphi\,dx
\end{equation}
for every $\varphi \in W^{1,p}_0(\Omega)$, i.e., if and only if it is a weak solution to the Dirichlet-type problem
\begin{equation}
\label{eq:THE_PDE}
\begin{cases}
\D f = \L |f|^{p-2}f, &\text{in } \Omega,\\
f=0, &\text{on } \partial\Omega.
\end{cases}
\end{equation}
In this case, $f$ is an eigenfunction of $\D$ in $W_0^{1,p}(\Omega)$ corresponding to the eigenvalue $\L$. Moreover, $\L$ is the smallest among the real Dirichlet eigenvalues of $\D$.
\end{proposition}
\begin{proof}
    The forward direction is specializing Proposition~\ref{p:eleq_q}. For the converse direction, suppose that $f\in W_0^{1,p}(\Omega)$ is a non-zero weak solution of the problem \eqref{eq:THE_PDE}. Taking $\varphi=f$ and using the $1$-homogeneity of $h_{L_{Q,p,f}}$, together with the identity
\[
\int_\Omega h_{L_{Q,p,f}}(\nabla f)^p\,dx=(\Ep f)^p,
\]
we obtain $(\Ep f)^p=\L\|f\|_{L^p(\Omega)}^p$. Consequently,
\[
\Rq(f)=\L,
\]
and hence $f$ is a minimizer. This proves the equivalence in the statement.

More generally, if $f\neq 0$ is an eigenfunction of $\D$ corresponding to a real eigenvalue $\lambda$, then the same argument gives
\[
\lambda
=
\frac{(\Ep f)^p}{\|f\|_{L^p(\Omega)}^p}
=
\Rq(f)
\geq \L.
\]
Therefore, $\L$ is the smallest among the real Dirichlet eigenvalues of $\D$.
\end{proof}

The next result concerns smoothness of minimizers.

\begin{proposition}\label{p:regularity2} Let $\Omega \subset \R^n$ be a bounded, open set, $Q\in \conbod[m]_{\operatorname{s}}$, and $p >1$. Let $f \in W_0^{1,p}(\Omega)$ be a weak solution to the problem 
\begin{equation}\label{e:prob}
\begin{split}
\begin{cases}
\D f = w(x) |f|^{p-2}f &\text{in } \Omega,\\
f =0 &\text{on } \partial \Omega,
\end{cases}
\end{split}
\end{equation}
 where $w \colon \Omega \to \R$ is a weight function. If $w \in  L^{\frac{n}{p}}(\Omega)$, then $f\in L^k(\Omega)$ for every $k \geq 1$.    
\end{proposition}

\begin{proof} Consider the quasilinear elliptic operator $\mathcal{L}$ from the proof of Proposition~\ref{p:regularity1} acting on $W_0^{1,p}(\Omega)$ for a fixed $f$. Then we can rewrite \eqref{e:prob} as 
\[
\begin{cases}
\mathcal{L} f = d\cdot w(x) |f|^{p-2}f &\text{in } \Omega,\\
f =0 &\text{on } \partial \Omega,
\end{cases}
\]
We already know that $\mathcal{L}$ satisfies Lemma~\ref{l:Tolksdorf}, so the result follows from \cite[Proposition~1.2]{GV88}.
\end{proof}

We are now in a position to prove Theorem~\ref{t:operatoranalysis}.

\begin{proof}[Proof of Theorem~\ref{t:operatoranalysis}]
First, according to Proposition~\ref{p:eleq}, a function $f_p\in W_0^{1,p}(\Omega)\setminus\{0\}$ minimizes $\L$ if and only if it is a weak solution to
\[
\begin{cases}
\D f = \L |f|^{p-2}f &\text{in } \Omega,\\
f =0 &\text{on } \partial \Omega.
\end{cases}
\]
By the same proposition, we have that $\L$ is the smallest among the real Dirichlet eigenvalues  of $\D$.

We next establish the regularity of an eigenfunction corresponding to $\L$. Suppose first that $1<p\leq n$. Since $f_p$ is a weak solution of the type appearing in Proposition~\ref{p:regularity2}, with the constant weight $w\equiv\L$, Proposition~\ref{p:regularity2} yields $f_p\in\bigcap_{k\geq 1}L^k(\Omega).$ Set
\[
h_0=\L |f_p|^{p-2}f_p.
\]
Choose $q>\frac{n}{p}$. Since $f_p\in L^{q(p-1)}(\Omega)$, we have $h_0\in L^q(\Omega)$, and Proposition~\ref{p:regularity1}(b) therefore implies that $f_p\in L^\infty(\Omega)$. If $p>n$, the boundedness of $f_p$ follows directly from the Sobolev embedding $W_0^{1,p}(\Omega)\subset L^\infty(\Omega)$. Thus, in either case, $h_0\in L^\infty(\Omega)$. Proposition~\ref{p:regularity1}(c) now gives $f_p\in C^{1,\alpha}(\Omega)$ for some $\alpha\in(0,1)$, while Proposition~\ref{p:regularity1}(d) gives $f_p\in C^{1,\alpha}(\overline{\Omega})$ provided $\partial\Omega$ is of class $C^{2,\alpha}$.

Finally, since $\Rq(f)=\Rq(|f|)$ for every $f\in W_0^{1,p}(\Omega)$, the function $g=|f_p|$ is also a minimizer of $\L$. By Proposition~\ref{p:eleq}, $g$ is therefore a weak solution of
\[
\D g=\L g^{p-1}\geq0.
\]
Moreover, by the regularity established above, $g\in C^{1,\alpha}(\Omega)$. For the fixed convex body $L_{Q,p,g}$, let
\[
\mathcal L u
=
-\operatorname{div}\left(
h_{L_{Q,p,g}}(\nabla u)^{p-1}
\nabla h_{L_{Q,p,g}}(\nabla u)
\right).
\]
Then
\[
\mathcal L g=\L g^{p-1}\geq0
\]
weakly. Thus, by Proposition~\ref{p:regularity1}, we can apply the strong maximum principle. Hence, if $\Omega$ is connected, either $g\equiv0$ or $g>0$ in $\Omega$. Since $f_p$ is an eigenfunction, $g\not\equiv0$, and therefore $|f_p|>0$ throughout $\Omega$. Since $f_p$ is continuous and $\Omega$ is connected, $f_p$ has a constant sign. This completes the proof.
\end{proof}

\section{Inequalities for the first eigenvalue}

\subsection{The $m$th-order affine $L^p$ Talenti inequality}
\label{s:talenti}

This section is dedicated to proving the $m$th-order affine $L^p$ Talenti inequality, Theorem~\ref{t:talenti}.

Throughout this section, we fix $p \in (1,n)$ and a bounded, open set $\Omega\subset \R^n$. We let $f_p$ be a minimizer of $\Rq$, or, alternatively by Proposition~\ref{p:eleq}, a solution to \eqref{eq:THE_PDE}, which exists by Theorem~\ref{t:existence}. We set $\mu_p(t):=\mu_{f_p}(t)$ to be the distribution function of $f_p$ over $\Omega$. It is also beneficial to set $p^\prime=\frac{p}{p-1}$, the H\"older conjugate of $p$, and similarly for $n$. With this notation, we are setting out to show
 \begin{equation}
    \left(n\omega_n^\frac{1}{n}(\mu_p(t))^\frac{1}{n^\prime}\right)^{p^\prime} \leq - \mu_p^\prime(t)\left(\L\left(\frac{\mu_p(t)}{t^{1-p}}+(p-1)\int_t^{+\infty}\frac{\mu_p(\tau)}{\tau^{2-p}}d\tau\right)\right)^\frac{p^\prime}{p}.
    \label{eq:Talenti_mth_again}
\end{equation}

We will again make use of the body $L_{Q,p,f_p}$ defined in \eqref{eq:L_body}, whose support function is given by \eqref{eq:L_body_supp}. By Jensen's inequality, we have that

\begin{align*}
& \left(\frac{\int_{\{t<f_p(x)<t+h\}}\left(h_{L_{Q,p,f_p}}\left(\nabla f_p(x)\right)\right)^p d x}{\int_{\{t<f_p(x)<t+h\}} h_{L_{Q,p,f_p}}\left(\nabla f_p(x)\right) d x}\right)^{-\frac{1}{p-1}} \\
&\quad=\left(\frac{\int_{\{t<f_p(x)<t+h\}}\left(h_{L_{Q,p,f_p}}\left(\nabla f_p(x)\right)\right)^{p-1}\left(h_{L_{Q,p,f_p}}\left(\nabla f_p(x)\right)\right) d x}{\int_{\{t<f_p(x)<t+h\}} h_{L_{Q,p,f_p}}\left(\nabla f_p(x)\right) d x}\right)^{-\frac{1}{p-1}} \\
& \quad \leq \frac{\int_{\{t<f_p(x)<t+h\}}\left(h_{L_{Q,p,f_p}}\left(\nabla f_p(x)\right)\right)^{-1} h_{L_{Q,p,f_p}}\left(\nabla f_p(x)\right) d x}{\int_{\{t<f_p(x)<t+h\}} h_{L_{Q,p,f_p}}\left(\nabla f_p(x)\right) d x} \\
& \quad=\frac{\int_{\{t<f_p(x)<t+h\}} d x}{\int_{\{t<f_p(x)<t+h\}} h_{L_{Q,p,f_p}}\left(\nabla f_p(x)\right) d x}.
\end{align*}

We obtain from this inequality the following bounds on the difference quotient of $\mu_p(t)$:
\begin{equation}
\label{eq:difference_quotient_bound}
\begin{split}
    &\left(\frac{\int_{\{t<f_p(x)<t+h\}}h_{L_{Q,p,f_p}}\left(\nabla f_p(x)\right) d x}{h}\right)^{p^\prime} 
    \\
    &\leq -\frac{\mu_p(t+h)-\mu_p(t)}{h}\left(\frac{\int_{\{t<f_p(x)<t+h\}}h_{L_{Q,p,f_p}}\left(\nabla f_p(x)\right)^p d x}{h}\right)^\frac{p^\prime}{p}.
    \end{split}
\end{equation}
By sending $h\to 0^+$ in \eqref{eq:difference_quotient_bound}, we obtain, introducing the notation $\Omega(t)=\Omega(t,f_p)$,
\begin{equation}
\label{eq:difference_quotients}
\begin{split}
    \left(-\frac{d}{dt}\int_{\Omega(t)}h_{L_{Q,p,f_p}}\left(\nabla f_p(x)\right) d x\right)^{p^\prime} 
    \leq -\mu^\prime_p(t)\left(-\frac{d}{dt}\int_{\Omega(t)}h_{L_{Q,p,f_p}}\left(\nabla f_p(x)\right)^p d x\right)^\frac{p^\prime}{p}.
    \end{split}
\end{equation}
By the coarea formula \eqref{eq:coarea}, we have 
\[
\int_{\Omega(t)}h_{L_{Q,p,f_p}}\left(\nabla f_p(x)\right) d x = \int_t^\infty P_{L_{Q,p,f_p}}(\Omega(s),\R^n) ds.
\]
Then, by the isoperimetric inequality \eqref{eq:min}, we obtain the inequality
\[
-\frac{d}{dt}\int_{\Omega(t)}h_{L_{Q,p,f_p}}\left(\nabla f_p(x)\right) d x = P_{L_{Q,p,f_p}}(\Omega(t),\R^n) \geq n\vol(L_{Q,p,f_p})^\frac{1}{n}\left(\mu_p(t)\right)^\frac{1}{n^\prime}.
\]
We then use the $m$th-order Busemann-Petty centroid inequality \eqref{eq:BPcLQ} to obtain that
\[
\vol(L_{Q,p,f_p}) \geq \omega_n,
\]
with equality when $f_p$ is radial. 

Inserting these computations into \eqref{eq:difference_quotients} yields the inequality
\begin{equation}
\label{eq:left_done}
    (n\omega_n^\frac{1}{n}\left(\mu_p(t)\right)^\frac{1}{n^\prime})^{p^\prime} \leq -\mu^\prime_p(t)\left(-\frac{d}{dt}\int_{\Omega(t)}h_{L_{Q,p,f_p}}\left(\nabla f_p(x)\right)^p d x\right)^\frac{p^\prime}{p}.
\end{equation}
To handle the right-hand side of \eqref{eq:left_done}, we make use of the fact that $f_p$ satisfies the equation \eqref{eq:weak_PDE_solved}. We set in that equation $\varphi = (f_p-t)_+ = \max\{f_p-t,0\}$
to obtain

\begin{equation}
\label{eq:supp_on_Omega_t}
\begin{split}
\int_{\Omega(t)}  h_{L_{Q,p,f_p}}(\nabla f_p)^{p} dx &=\int_{\Omega(t)}  h_{L_{Q,p,f_p}}(\nabla f_p)^{p-1} \langle \nabla h_{L_{Q,p,f_p}}(\nabla f_p), \nabla f_p \rangle dx 
\\
&= \L \int_{\Omega(t)} |f_p|^{p-2} f_p(f_p-t)dx,
\end{split}
\end{equation}
where we used \eqref{eq:support_homo} for the first equality. 

Finally, we claim that 
\begin{equation}
    -\frac{d}{dt}\int_{\Omega(t)} |f_p|^{p-2} f_p(f_p-t)dx \leq \frac{\mu_p(t)}{t^{1-p}}+(p-1)\int_t^{+\infty}\frac{\mu_p(\tau)}{\tau^{2-p}}d\tau.
    \label{eq:final_bound}
\end{equation}

In light of \eqref{eq:supp_on_Omega_t} and \eqref{eq:left_done}, the bound \eqref{eq:final_bound} will establish \eqref{eq:Talenti_mth_again}. We will follow the Euclidean case from \cite{BT99} closely. First, we observe that the integral over $\Omega(t)$ in \eqref{eq:final_bound} is absolutely continuous in $t$. Consequently, by writing
\begin{equation}
\label{eq:quotients_level}
\begin{split}
    &-\left(\frac{\int_{\Omega(t+h)}\left(f_{p}(x)\right)^{p-1}\left(f_{p}(x)-t-h\right) d x-\int_{\Omega(t)}\left(f_{p}(x)\right)^{p-1}\left(f_{p}(x)-t\right) d x}{h}\right)
    \\
    & =\frac{\int_{\Omega(t) \setminus \Omega(t+h)}\left(f_{p}(x)\right)^{p-1}\left(f_{p}(x)-t-h\right) d x}{h}+ \int_{\Omega(t)}\left(f_{p}(x)\right)^{p-1} d x.
    \end{split}
\end{equation}
Sending $h\to 0^+$ in \eqref{eq:quotients_level}, and throwing away the first term in the second line (which is negative), we obtain 
\begin{equation}
\label{eq:deriv_dis_bound}
    -\frac{d}{dt}\int_{\Omega(t)} |f_p|^{p-2} f_p(f_p-t)dx \leq \int_{\Omega(t)}\left(f_{p}(x)\right)^{p-1} d x .
\end{equation}
Finally, \eqref{eq:deriv_dis_bound}, together with the identity
\[
\int_{\Omega(t)}\left(f_{p}(x)\right)^{p-1} d x = \frac{\mu_p(t)}{t^{1-p}}+(p-1)\int_t^{+\infty}\frac{\mu_p(\tau)}{\tau^{2-p}}d\tau,
\]
which follows from integration by parts, yields \eqref{eq:final_bound}. Thus, we have established \eqref{eq:Talenti_mth_again}, i.e. we have proven Theorem~\ref{t:talenti}.

\subsection{The $m$th-order affine Faber-Krahn inequality}
\label{sec:fk}

This section is dedicated to proving the $m$th-order affine Faber-Krahn inequality, Theorem~\ref{t:fk}.

Let $B = B_\Omega$ denote the closed ball centered at the origin such that $\vol(B) = \vol(\Omega)$. Throughout the proof we will make frequent use of Proposition~\ref{p:properties}(c). 

We will begin with the case $p=1$. For $r >0$, we have 
\[
k_r = \frac{\||\nabla \chi_{rB_2^n}|\|_{L^1(\R^n)}}{\|\chi_{rB_2^n}\|_{L^1(\R^n)}} = \frac{\text{vol}_{n-1}(\partial (rB_2^n))}{\vol(rB_2^n)} = \frac{n}{r} = \frac{n \omega_n^{\frac{1}{n}} }{\vol(rB_2^n)^{\frac{1}{n}}}
\]
in the sense of $\operatorname{BV}(\R^n)$. Since $\chi_{rB_2^n}$  is rotationally invariant, we can approximate $\chi_{rB_2^n}$ by a sequence $(\varphi_i)$ smooth radially symmetric functions such that $\||\nabla \varphi_i|\|_{L^1(\R^n)} \to \||\nabla \chi_{rB_2^n}|\|_{L^1(\R^n)}$. However, since $\mathcal{E}_{Q,1} f = \||\nabla f|\|_{L^1(\R^n)}$ for any reasonably smooth radial function $f$ (see \cite[Theorem 1.2]{LRZ24}), it follows that 
\[
\lim_{i \to \infty} \frac{\mathcal{E}_{Q,1} \varphi_i}{\|\varphi_i\|_{L^1(\R^n)}} = \lim_{i \to \infty} \frac{\||\nabla \varphi_i|\|_{L^1(\R^n)}}{\|\varphi_i\|_{L^1(\R^n)}} = k_r.  
\]
Consequently, we have 
\[
\lambda_{1,1}^{\mathcal{A}}(Q,B) \leq \lambda_{1,1}(B) \leq \frac{n\omega_n^{1/n}}{\vol(B)^{\frac{1}{n}}}. 
\]
Assume that $f_1 \in \operatorname{BV}(\Omega)$ is a minimizer of $\lambda_{1,1}^{\mathcal{A}}(Q,\Omega)$. By H\"older's inequality paired with the $m$th-order affine Sobolev inequality (Lemma~\ref{l;Sobolevlemma}(a)), it follows that 
\begin{equation}\label{e:fk1}
\begin{split}
\lambda_{1,1}^{\mathcal{A}}(Q,\Omega)
&\geq \frac{\mathcal{E}_{Q,1} f_1}{\|f_1\|_{L^1(\Omega)}}
\geq \frac{\mathcal{E}_{Q,1}f_1}
{\|f_1\|_{L^{\frac{n}{n-1}}(\Omega)}\vol(\Omega)^{\frac{1}{n}}} \geq \frac{n\omega_n^{\frac{1}{n}}}{\vol(\Omega)^{\frac{1}{n}}}
= \frac{n\omega_n^{\frac{1}{n}}}{\vol(B)^{\frac{1}{n}}}
\geq \lambda_{1,1}^{\mathcal{A}}(Q,B).
\end{split}
\end{equation}

Suppose that $\lambda_{1,1}^{\mathcal{A}}(Q,\Omega) = \lambda_{1,1}^{\mathcal{A}}(Q,B)$. Then there must be equality in \eqref{e:fk1}, which means the equality conditions of the $m$th-order affine Sobolev inequality kick in and force $f_1 = a \chi_{E}$ for some ellipsoid $E \subset \R^n$. Finally, under the equality conditions in H\"older's inequality, we obtain $\Omega = E$ up to a null set, and hence $\Omega$ is an ellipsoid, as desired. 

Now we assume that $p >1$. Suppose that $f_p \in W^{1,p}_0(\Omega)$ is a minimizer of $\L$, and let $f_p^\star \in W_0^{1,p}(B)$ be its spherically symmetric rearrangement. By the $m$th-order P\'olya-Szeg\H{o} principle (Lemma~\ref{l;Sobolevlemma}(b)), we have 
\begin{equation}\label{e:fk2}
\begin{split}
\L \geq \frac{(\Ep f_p)^p}{\|f_p\|_{L^p(\Omega)}^p} \geq \frac{(\Ep f_p^\star)^p}{\|f_p^\star\|_{L^p(B)}^p} \geq \lambda_{1,p}^{\mathcal{A}}(Q,B).
\end{split}
\end{equation}

If $\L = \lambda_{1,p}^{\mathcal{A}}(Q,B)$, then we obtain equality in \eqref{e:fk2}, forcing $f_p^\star$ to be a minimizer of $\lambda_{1,p}^{\mathcal{A}}(Q,B)$, and since $\Ep f_p^\star = \||\nabla f_p^\star|\|_{L^p(B)}$ and 
\[
\lambda_{1,p}(B) \geq \lambda_{1,p}^{\mathcal{A}}(Q,B) = \frac{(\Ep f_p^\star)^p}{\|f_p^\star\|_{L^p(B)}^p} = \frac{\||\nabla f_p^\star|\|_{L^p(B)}^p}{\|f_p^\star\|_{L^p(B)}^p},
\]
we also have that $f_p^\star$ is a minimizer of $\lambda_{1,p}(B)$. Because $f_p^\star$ is a first positive eigenfunction of the $p$-Laplace operator on $B$, $f_p^\star$ is strictly radially decreasing. By the equality conditions of the $m$th-order affine P\'olya-Szeg\H{o} principle, it is then the case that $f_p(x) = F(|Ax|)$ for some $F \colon \R \to [0,\infty)$ and some invertible linear transformation $A \colon \R^n \to \R^n$. 

Finally, we need only show that $A \Omega$ is a ball in $\R^n$. Without loss of generality, we may assume that $\det(A) = 1$. Set $g(x) = f_p(A^{-1} x)$. The function $g$ is radial, and so it follows that 
\[
\frac{\||\nabla g|\|_{L^p(A\Omega)}^p}{\|g\|_{L^p(A\Omega)}^p} = \frac{(\Ep g)^p}{\|g\|_{L^p(A\Omega)}^p} = \lambda_{1,p}^{\mathcal{A}}(Q,A\Omega) \leq \lambda_{1,p}(A\Omega). 
\]
Therefore, $g$ is a minimizer of $\lambda_{1,p}(A\Omega)$. Since $f_p^\star$ is a minimizer of $\lambda_{1,p}(B)$, we conclude that 
\[
\lambda_{1,p}(A\Omega) = \frac{(\Ep g)^p}{\|g\|_{L^p(A\Omega)}^p} = \frac{(\Ep f_p^\star)^p}{\|f_p^\star\|_{L^p(B)}^p} = \lambda_{1,p}(B). 
\]
This, together with the equality conditions of the classical Faber-Krahn inequality for the $p$-Laplace operator imply that $A\Omega = B+x_0$ for some $x_0\in \R^n$, or in other words, $\Omega$ is an ellipsoid, as claimed. 

\subsection{Cheeger sets}
\label{sec:cheeger}
In this section, we prove Theorem~\ref{t:cheeger}. Let $f\in \operatorname{BV}(\Omega)$, and, for $t>0$, set
\[
\Omega_+(t)=\{x\in\R^n:f(x)>t\},
\qquad
\Omega_-(t)=\{x\in\R^n:f(x)<-t\}.
\]
From the coarea formula \eqref{eq:coarea_BV} for functions of bounded variation and the fact that
$n_C=-\eta_{\chi_C}$, we have
\begin{align*}
\int_{\R^n}h_Q(\theta^t\eta_f(x))d|Df|(x)
&=
\int_0^\infty
\|-\theta\|_{\Pi^\circ_{Q,1}\Omega_+(t)}dt
+
\int_0^\infty
\|\theta\|_{\Pi^\circ_{Q,1}\Omega_-(t)}dt.
\end{align*}
Therefore, by Minkowski's integral inequality, we have
\begin{align*}
\E f
&=
d_{n,1}(Q)
\left(
\int_{\S}
\left(
\int_0^\infty
\|-\theta\|_{\Pi^\circ_{Q,1}\Omega_+(t)}dt
+
\int_0^\infty
\|\theta\|_{\Pi^\circ_{Q,1}\Omega_-(t)}dt
\right)^{-nm}
d\theta
\right)^{-\frac1{nm}}
\\
&\geq
d_{n,1}(Q)
\int_0^\infty
\left(
\int_{\S}
\|-\theta\|_{\Pi^\circ_{Q,1}\Omega_+(t)}^{-nm}
d\theta
\right)^{-\frac1{nm}}dt
+
d_{n,1}(Q)
\int_0^\infty
\left(
\int_{\S}
\|\theta\|_{\Pi^\circ_{Q,1}\Omega_-(t)}^{-nm}
d\theta
\right)^{-\frac1{nm}}dt.
\end{align*}
Using the change of variables $\theta\mapsto-\theta$ on $\S$ in the first integral, we obtain
\begin{align*}
\E f
&\geq
d_{n,1}(Q)
\int_0^\infty
\left(
\int_{\S}
\|\theta\|_{\Pi^\circ_{Q,1}\Omega_+(t)}^{-nm}
d\theta
\right)^{-\frac1{nm}}dt
+
d_{n,1}(Q)
\int_0^\infty
\left(
\int_{\S}
\|\theta\|_{\Pi^\circ_{Q,1}\Omega_-(t)}^{-nm}
d\theta
\right)^{-\frac1{nm}}dt
\\
&=
\int_0^\infty \E\chi_{\Omega_+(t)}dt
+
\int_0^\infty \E\chi_{\Omega_-(t)}dt.
\end{align*}
On the other hand, by the layer-cake formula,
\[
\|f\|_{L^1(\Omega)}
=
\int_0^\infty \vol(\Omega_+(t))dt
+
\int_0^\infty \vol(\Omega_-(t))dt.
\]
Supposing that $f$ is a minimizer of $\lambda_{1,1}^\mathcal{A}(Q,\Omega)$, we have
\begin{align*}
0
&=
\E f
-
\lambda_{1,1}^\mathcal{A}(Q,\Omega)\|f\|_{L^1(\Omega)}
\\
&\geq
\int_0^\infty
\left(
\E\chi_{\Omega_+(t)}
-
\lambda_{1,1}^\mathcal{A}(Q,\Omega)\vol(\Omega_+(t))
\right)dt
\\
&\qquad
+
\int_0^\infty
\left(
\E\chi_{\Omega_-(t)}
-
\lambda_{1,1}^\mathcal{A}(Q,\Omega)\vol(\Omega_-(t))
\right)dt
\geq 0.
\end{align*}
Thus, we can find $t>0$ such that at least one of $\Omega_\pm(t)$ has positive volume and the corresponding integrand vanishes. Hence, the corresponding $\chi_{\Omega_\pm(t)}$ is a minimizer of $\lambda_{1,1}^\mathcal{A}(Q,\Omega)$.

For the final claim, we use the known affine invariance for
$\Pi^\circ_{Q,1}$, \cite[Proposition 3.6]{HLPRY25_2}, to obtain, for every $A\in \operatorname{GL}(n)$, $x_0\in \R^n$ and $K\subset\R^n$ a set of finite perimeter,
\begin{align*}
\mathcal{R}_{Q,1}^\mathcal{A}\chi_{(AK+x_0)}
&=
\frac{\E \chi_{(AK+x_0)}}{\|\chi_{(AK+x_0)}\|_{L^1(\Omega)}}
=
d_{n,1}(Q)(nm)^{-\frac{1}{nm}}
\frac{\vol[nm](\Pi^\circ_{Q,1}(AK+x_0))^{-\frac{1}{nm}}}
{\vol(AK+x_0)}
\\
&=
|\det(A)|^{-\frac{1}{n}}
d_{n,1}(Q)(nm)^{-\frac{1}{nm}}
\frac{\vol[nm](\Pi^\circ_{Q,1}K)^{-\frac{1}{nm}}}
{\vol(K)}
\\
&=
|\det(A)|^{-\frac{1}{n}}
\mathcal{R}_{Q,1}^\mathcal{A}\chi_K,
\end{align*}
and then the claim follows, as this shows that the Rayleigh quotient decreases when the volume of $K$ is increased. \hfill \qed

\vspace{0.5cm}

{\bf Declaration of competing interest:} The authors declare that they have no competing financial interests or personal relationships concerning this work.

\vspace{0.1cm}

{\bf Funding:}  D.L. was supported by the U.S. National Science Foundation under grant DMS-2502744. M.R. was supported by the U.S. National Science Foundation under grant DMS-2548742.

\bibliographystyle{acm}

\end{document}